\documentclass[12pt]{extarticle}
\usepackage{amsmath, amsthm, amssymb, mathtools, hyperref, color}
\usepackage[shortlabels]{enumitem}
\usepackage{graphicx}
\usepackage{adjustbox}
\usepackage{tikz}
\usetikzlibrary{arrows,calc,fit,matrix,positioning,cd}

\oddsidemargin \evensidemargin
\newtheorem{theorem}{Theorem}
\numberwithin{theorem}{section}
\newtheorem{proposition}[theorem]{Proposition}
\newtheorem{lemma}[theorem]{Lemma}
\newtheorem{corollary}[theorem]{Corollary}
\newtheorem{conjecture}[theorem]{Conjecture}
\newtheorem{challenge}[theorem]{Challenge}
\newtheorem{problem}[theorem]{Problem}

\theoremstyle{definition}

\newtheorem{remark}[theorem]{Remark}
\newtheorem{example}[theorem]{Example}

\newcommand{\RR}{\mathbb{R}}
\newcommand{\QQ}{\mathbb{Q}}
\newcommand{\ZZ}{\mathbb{Z}}
\newcommand{\PP}{\mathbb{P}}
\newcommand{\CC}{\mathbb{C}}
\date{}

\allowdisplaybreaks[1]

\definecolor{darkblue}{RGB}{0,0,160}
\hypersetup{
colorlinks,%
citecolor=black,%
filecolor=black,%
linkcolor=darkblue,%
urlcolor=darkblue
}

\title{\textbf{The Geometry of Gaussoids}
\author{Tobias Boege, Alessio D'Al\`i, Thomas Kahle and Bernd Sturmfels} }

\begin{document}
\maketitle

\vspace{-5ex}
\begin{center}
\textit{Dedicated to the memory of Franti\v{s}ek Mat\'u\v{s}}
\end{center}

\begin{abstract}
\noindent
A gaussoid is a combinatorial structure that encodes independence in
probability and statistics, just like matroids encode independence in
linear algebra.  The gaussoid axioms of Ln\v{e}ni\v{c}ka and Mat\'u\v{s}
are equivalent to compatibility with certain quadratic relations among
principal and almost-principal minors of a symmetric matrix.  We
develop the geometric theory of gaussoids, based on the Lagrangian
Grassmannian and its symmetries.  We introduce oriented gaussoids and
valuated gaussoids, thus connecting to real and tropical geometry.  We
classify small realizable and non-realizable gaussoids.  Positive
gaussoids are as nice as positroids: they are all realizable via
graphical~models.
\end{abstract}

\section{Introduction}

Gaussoids are combinatorial structures that arise in statistics, and
are reminiscent of matroids. They were introduced by Ln\v{e}ni\v{c}ka
and Mat\'u\v{s}~\cite{LM} to represent conditional independence
relations among $n$ Gaussian random variables.  The theory of matroids
is ubiquitous in the mathematical sciences, as it captures the
combinatorial essence of many objects in algebra and geometry.
Matroids of rank $d$ on $[n] = \{1,2,\ldots,n\}$ are possible supports
of Pl\"ucker coordinates on the Grassmannian of $d$-dimensional linear
subspaces in a vector space~$K^n$.

This article develops the geometric theory of gaussoids, with a focus on
parallels to matroid theory.  The role of the Grassmannian is played
by a natural projection of the Lagrangian Grassmannian, namely the
variety of principal and almost-principal minors of a symmetric
$n \times n$-matrix $\Sigma$.  Gaussoids aim to characterize which
almost-principal minors can simultaneously vanish provided $\Sigma$ is
positive definite. This issue is important in statistics, where
$\Sigma$ is the covariance matrix of a Gaussian distribution on
$\RR^n$, and almost-principal minors measure partial correlations. The
sign of a minor indicates whether the partial correlation is positive
or negative. The minor is zero if and only if conditional independence~holds.

Our goal in this paper is to carry out the program that was suggested in~\cite[\S 4]{StuB}.
We assume that our readers are familiar with the geometric approach to
matroids, including oriented matroids and valuated matroids, as well
as basic concepts in algebraic statistics.  Introductory books for the
former include \cite{OM, BGW, MS}.  Sources for the latter
include~\cite{DSS, Stu, StuB}.

Let $\Sigma = (\sigma_{ij})$ be a symmetric $n \times n$-matrix whose
$\binom{n+1}{2}$ entries are unknowns.  A {\em minor} of $\Sigma$ is
the determinant of a square submatrix.  The projective variety
parametrized by all minors of $\Sigma$ is the {\em Lagrangian
Grassmannian} ${\rm LGr}(n,2n)$.  It is obtained by intersecting the
usual Grassmannian ${\rm Gr}(n,2n)$ in its Pl\"ucker embedding in
$\PP^{\binom{2n}{n}-1}$ with a linear subspace.  An affine chart of
${\rm LGr}(n,2n)$ consists of all row spaces of rank $n$ matrices of
the form \setcounter{MaxMatrixCols}{20}
\begin{equation}
\label{eq:IdSigma} \bigl( \,\,{\rm Id}_n \, \,\, \Sigma \,\, \bigr) \quad = \quad
\begin{pmatrix}
1 & 0 & 0 & \cdots & 0 & & \sigma_{11} & \sigma_{12} & \sigma_{13} & \cdots & \sigma_{1n} \\
0 & 1 & 0 & \cdots & 0 & & \sigma_{12} & \sigma_{22} & \sigma_{23} & \cdots & \sigma_{2n} \\
0 & 0 & 1 & \cdots & 0 & & \sigma_{13} & \sigma_{23} & \sigma_{33} & \cdots & \sigma_{3n} \\
\vdots & \vdots & \vdots & \ddots & \vdots &  &  \vdots & \vdots & \vdots & \ddots & \vdots  \\
0 & 0 & 0 & \cdots & 1 & & \sigma_{1n} & \sigma_{2n} & \sigma_{3n} &
\cdots & \sigma_{nn}
\end{pmatrix}.
\end{equation}
The right $n \times n$-block is symmetric.  The quadratic Pl\"ucker
relations for ${\rm Gr}(n,2n)$ restrict to quadrics that define
${\rm LGr}(n,2n)$. It is known that those quadrics form a Gr\"obner
basis.  For more information we refer to Oeding's dissertation
\cite[\S~III.A]{lukediss} and the references therein.

A minor of $\Sigma$ is {\em principal} if its row indices and its
column indices coincide, and it is {\em almost-principal} if its row
and column indices differ in exactly one element.  We introduce
unknowns that represent the $2^n$ principal minors and the
$2^{n-2} \binom{n}{2}$ almost-principal minors:
\[
\mathcal{P} = \bigl\{\,p_I \,:\,I \subseteq [n]\bigr\} \qquad
\text{and} \qquad \mathcal{A} = \big\{a_{ij|K} \,: \, i,j \in [n]
\,\,\hbox{distinct and}\, \,K \subseteq [n] \backslash \{i,j\}
\bigr\}.
\]
To simplify notation, we write $p$ for $p_\emptyset$, $p_{12}$ for
$p_{\{1,2\}}$, $a_{12|3}$ for $a_{12|\{3\}}$, etc.  These unknowns correspond respectively to
the vertices and $2$-faces of the $n$-cube, as shown in
Figure~\ref{fig:eins}.  By convention, $p = 1$, and this
variable serves as a homogenization variable.
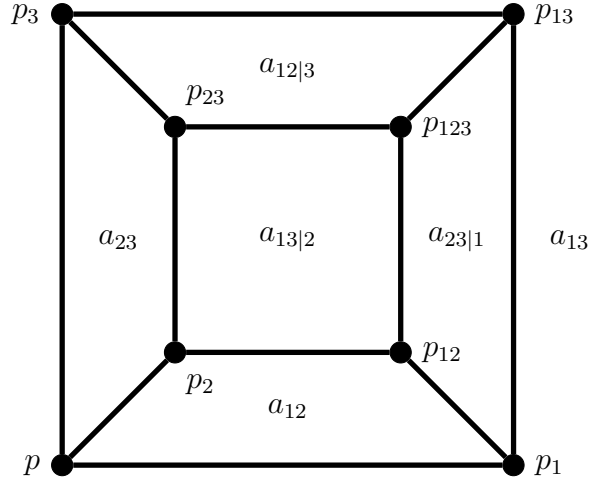
\begin{figure}[ht]
  \begin{center}
    \begin{tikzpicture}[scale=1.5]
      \tikzset{vertex/.style = {shape=circle,fill=black,draw,minimum size=8pt, inner sep=0pt}}
      \tikzset{edge/.style = {line width=2pt}}

      \node [vertex, label=left:$p$] (1) at (0,0) {};
      \node [vertex, label=right:$p_{1}$] (2) at (4,0) {};
      \node [vertex, label=right:$p_{13}$] (3) at (4,4) {};
      \node [vertex, label=left:$p_3$] (4) at (0,4) {};
      \node [vertex, label={-87:$p_{2}$} ] (5) at (1,1) {};
      \node [vertex, label=right:$p_{12}$] (6) at (3,1) {};
      \node [vertex, label=right:$p_{123}$] (7) at (3,3) {};
      \node [vertex, label={87:$p_{23}$}] (8) at (1,3) {};

      \draw [edge] (1) -- (2) -- (3) -- (4) -- (1) ;
      \draw [edge] (5) -- (6) -- (7) -- (8) -- (5) ;
      \draw [edge] (1) -- (5);
      \draw [edge] (2) -- (6);
      \draw [edge] (3) -- (7);
      \draw [edge] (4) -- (8);

      \node (9) at (.5,2) {$a_{23}$};
      \node (10) at (3.5,2) {$a_{23|1}$};
      \node (11) at (4.5,2) {$a_{13}$};
      \node (12) at (2,2) {$a_{13|2}$};
      \node (13) at (2,.5) {$a_{12}$};
      \node (14) at (2,3.5) {$a_{12|3}$};
    \end{tikzpicture}
  \end{center}
  \vspace{-0.13in}
  \caption{\label{fig:eins} The vertices and $2$-faces
    of the $n$-cube are labeled by the set of unknowns $\mathcal{P} \cup \mathcal{A}$.
    \label{fig:3cube} }
\end{figure}\vspace{-2ex}

Consider the homomorphism
$\,\RR[\mathcal{P} \cup \mathcal{A} ] \rightarrow \RR[ \Sigma]\,$ from
a polynomial ring in $2^{n-2}(4+ \binom{n}{2})$ unknowns to a
polynomial ring in $\binom{n+1}{2}$ unknowns, where $p_I$ is mapped to
the minor of $\Sigma$ with row indices $I$ and column indices $I$, and
$a_{ij|K}$ is mapped to the minor of $\Sigma$ with row indices
$\{i\} \cup K$ and column indices $\{j\} \cup K$.  Here, the row
indices are sorted so that $i$ comes first and is followed by
 $K$, and the column indices are sorted so that $j$ comes first and is followed
by $K$, where the elements of $K$ are listed in increasing numerical order.
For instance, $ a_{12|3} $ maps to
$ \sigma_{12} \sigma_{33}- \sigma_{13} \sigma_{23} $ whereas
$ a_{13|2}$ maps to
$ -(\sigma_{12} \sigma_{23} - \sigma_{13} \sigma_{22})$. Maintaining
this sign convention is important to keep the algebra consistent with
its statistical interpretation.

Let $J_n$ denote the ideal generated by all homogeneous polynomials in
the kernel of the map above.  This defines an irreducible variety
$V(J_n)$ of dimension $\binom{n+1}{2}$ in the projective space
$\PP^{2^{n-2}(4+ \binom{n}{2})-1}$ whose coordinates are
$\mathcal{P} \cup \mathcal{A}$.  There is a natural projection from
${\rm LGr}(n,2n)$ onto $V(J_n)$, obtained by deleting all minors that
are neither principal nor almost-principal.  This is analogous to
\cite[Observation~III.12]{lukediss}, where the focus was on principal
minors~$p_I$.

\begin{proposition}
The degree of the projective variety of principal and almost-principal
minors coincides with the degree of the Lagrangian Grassmannian. For
$n \geq 2$, it equals
\[
  {\rm degree}(V(J_n)) \,\, =\,\,
  {\rm degree}({\rm LGr}(n,2n)) \,\,=\,\,
  \frac{ \binom{n+1}{2} ! }{1^{n} \cdot  3^{n-1} \cdot 5^{n-2}\,
    \cdots\, (2n-1) }.
\]
\end{proposition}

\begin{proof}
The degree of ${\rm LGr}(n,2n)$ is due to Hiller
\cite[Corollary~5.3]{Hil}.  We learned this formula from Totaro's comment
on the sequence A005118 in the OEIS:
$ 2, 16$, $ 768$, $ 292864, \ldots$

It suffices to show that the birational map from ${\rm LGr}(n,2n)$
onto $V(J_n)$ is base-point free.  There is an affine cover of
${\rm LGr}(n,2n)$ using charts that, after permuting coordinates, all
look like that in~\eqref{eq:IdSigma}.  In such a chart, the center of
the map
$\PP^{\binom{2n}{n}-1} \dashrightarrow
\PP^{2^{n-2}(4+ \binom{n}{2})-1} $ consists of the points whose coordinates indexed
by principal and almost-principal minors are all zero.  No such point
arises from a nonzero~$\Sigma$.
Therefore the center is disjoint
from~${\rm LGr}(n,2n)$.
\end{proof}

There are two natural symmetry classes of trinomials in $J_n$.  First,
there is one trinomial for each $2$-face of the $n$-cube. The
cardinality of that class is $2^{n-2} \binom{n}{2}$.  A
representative~is
\begin{equation} \label{eq:trinomial2face}
a_{12}^2 - p_1 p_2   + p_{12} p.
\end{equation}
Second, there is one trinomial for each inclusion of an edge in a
$3$-cube, in the boundary of the $n$-cube.  The number of these {\em
edge trinomials} is $ 12 \cdot 2^{n-3} \binom{n}{3}$. One representative is
\begin{equation} \label{eq:trinomialedge} p a_{23|1} - p_{1} a_{23} + a_{12} a_{13}.
\end{equation}

In Section~\ref{s:axioms} we review the axiom system for gaussoids
found in~\cite{LM}, and we show in Theorem~\ref{thm:gaussoid} that
these axioms are equivalent to compatibility with the edge
trinomials~\eqref{eq:trinomialedge}.  In Section~\ref{s:symmetry} we
examine a natural action of the group $G = {\rm SL}_2(\RR)^n$ on the
polynomial ring $\RR[\mathcal{P} \cup \mathcal{A}]$. This fixes the
ideal~$J_n$.  Certain finite subgroups of $G$ serve as symmetry groups
for the combinatorial structures in this paper.  In
Section~\ref{s:census} we classify gaussoids up to
$n = 5$, taking into account the various symmetry groups in~$G$.
Our computations make extensive use of state-of-the-art {\em SAT
solvers}.  In Section~\ref{s:orientedAndPos} we introduce and classify
oriented gaussoids.  Theorem~\ref{thm:posGaussoids} asserts that every
positive gaussoid is realizable by an undirected graphical model.
In Section~\ref{s:quadrics} we determine all quadrics in $J_n$,
and we conjecture that they generate.
Section~\ref{s:valuatedG} focuses
on valuated gaussoids and tropical geometry, and
Section~\ref{s:realizable} addresses the realizability problem for
gaussoids and oriented gaussoids.  Our supplementary materials website
\href{https://www.gaussoids.de}{www.gaussoids.de} contains various
classifications reported in this paper.

\section{Gaussians and Axioms}\label{s:axioms}
A symmetric $n \times n$-matrix $\Sigma = (\sigma_{ij})$ is the
covariance matrix of an $n$-dimensional normal (or {\em Gaussian})
distribution if $\Sigma$ is positive definite, i.e., if the $2^n$
principal minors $p_I$ of $\Sigma$ are all positive.  Let
$X_1,X_2,\ldots,X_n$ be random variables whose joint distribution is
Gaussian with covariance matrix~$\Sigma$.  For any subset
$K \subseteq [n] $ we write $X_K$ for the random vector
$(X_i: i \in K)$ in $\RR^{|K|}$.  The variable $X_i$ is independent of
the variable $X_j$ given the variable $X_K$ if and only if the
almost-principal minor $a_{ij|K}$ of $\Sigma$ is zero.  See
\cite[Proposition~3.1.13]{DSS}.
This {\em conditional independence (CI) statement} is usually denoted
by $ X_i \! \perp \!\!\! \perp \! X_j \,|\, X_K $ and also known as an
{\em elementary CI statement}.  Restriction to only these statements
is justified in \cite[\S~2.2.3]{Stu}.  Other notations found in the
literature include $i \! \perp \! j | K$,
$\langle i,j | K \rangle $, and $(ij|K)$.  We shall keep things simple
by identifying all of these symbols with our unknown
$a_{ij|K} \in \mathcal{A}$.

Reasoning and inference with conditional independence statements plays
a fundamental role in statistics, especially in the study of graphical
models \cite{DX,LM,MUYW,Stu,Sul}.  A guiding problem has been to
characterize collections of conditional independence statements that
can hold simultaneously within some class of distributions.  This led
to the theory of {\em semi\-graphoids}; see e.g.~\cite[\S~2]{MUYW}.  We
here focus on the class of Gaussian distributions on~$\RR^n$.  The
guiding problem now takes the following algebraic form: which sets of
almost-principal minors $a_{ij|K}$ can be simultaneously zero for a
positive definite symmetric $n \times n$-matrix~$\Sigma$?

To study this question, Ln\v{e}ni\v{c}ka and Mat\'u\v{s}~\cite{LM}
introduced the following axiom system, which we present here in our
notation.  As before, $\mathcal{A}$ is the set of all symbols
$a_{ij|K}$ where $i,j$ are distinct elements in
$[n] = \{1,2,\ldots,n\}$ and $K$ is a subset of
$[n] \backslash \{i,j\}$. Thus the set $\mathcal{A}$ consists of
$\binom{n}{2} 2^{n-2}$ symbols $a_{ij|K}$.  We identify these symbols
with the $2$-faces of the $n$-cube.

\smallskip Following \cite[Definition~1]{LM}, a subset ${\mathcal G}$
of $\mathcal{A}$ is called a {\em gaussoid} on $[n]$ if it satisfies the following four conditions for
 all pairwise distinct $i,j,k\in [n]$ and  all $L\subseteq [n] \backslash \{i,j,k\}$:
\begin{enumerate}[label=(G\arabic*),ref=(G\arabic*)]
\item\label{it:G1} $\{ a_{ij|L}, a_{ik|jL} \} \subset \mathcal{G}\,$ implies
$\,\{ a_{ik|L}, a_{ij|kL}\} \subset \mathcal{G}$,
\item\label{it:G2} $\{ a_{ij|kL} ,a_{ik|jL}  \} \subset \mathcal{G}\,$ implies
$\,\{ a_{ij|L}, a_{ik|L}\} \subset \mathcal{G}$,
\item\label{it:G3} $\{ a_{ij|L}, a_{ik|L}
  \} \subset \mathcal{G}\,$ implies
$\,\{ a_{ij|kL}, a_{ik|jL}\} \subset \mathcal{G}$,
\item\label{it:G4} $\{ a_{ij|L}, a_{ij|kL}  \} \subset \mathcal{G}\,$ implies
$\,\bigl( a_{ik|L} \in  \mathcal{G}\,$ or
$\, a_{jk|L} \in  \mathcal{G} \bigr)$.
\end{enumerate}
Axiom~\ref{it:G1} is the definition of a {\em semigraphoid}, and
\ref{it:G2} is known as the {\em intersection axiom}.  Axiom
\ref{it:G3} is a {\em converse to intersection}, and axiom \ref{it:G4}
is called {\em weak transitivity}.

Being a gaussoid is a necessary condition for a subset
$\mathcal{G} \subseteq \mathcal{A}$ to comprise the vanishing
almost-principal minors of a positive definite symmetric
$n \times n$-matrix $\Sigma$.  The gaussoid $\mathcal{G}$ is called {\em
realizable} if such a matrix $\Sigma$ exists.  All gaussoids are
realizable for $n=3$. This is no longer true for $n \geq 4$,  as shown in~\cite{DX, LM}.
For an explicit example see
 Remark \ref{rem:RealNull} below.

\begin{example}\label{ex:elevengaussoids} Let $n=3$. The set
$\mathcal{A}$ has $6$ elements, and hence it has $2^6=64$ subsets.  Among
these $64$ subsets, precisely $11$ are gaussoids. They are
\begin{equation}
\label{eq:elevengaussoids}
\begin{matrix}
\{\},\,
\{a_{12}\},\,
\{a_{13}\},\,
\{a_{23}\},\,
\{a_{12|3}\},\,
\{a_{13|2}\},\,
\{a_{23|1}\},
\{a_{12}, a_{12|3}, a_{13}, a_{13|2}\}, \\
\{a_{12}, a_{12|3}, a_{23}, a_{23|1}\},
\{a_{13}, a_{13|2}, a_{23}, a_{23|1}\},
\{a_{12}, a_{12|3}, a_{13}, a_{13|2}, a_{23}, a_{23|1}\}.
\end{matrix}
\end{equation}
Each of these gaussoids $\mathcal{G}$ is realizable by a positive
definite symmetric $3 \times 3$-matrix.  Equivalently, the variety
$V(J_3)$ contains a real point $(p,a)$ whose coordinates $p_I$ are all
positive and whose coordinates that vanish are precisely the elements
$a_{ij|K} $ in $\mathcal{G}$.  We invite the reader to check that all
$11$ gaussoids $\mathcal{G}$ arise from an appropriate point $(p,a)$
in the variety $V(J_3)$.  \hfill $\diamondsuit$
\end{example}

Gaussoids are analogous to matroids.  In matroid theory, one
asks which sets of maximal minors of a rectangular matrix can be
simultaneously nonzero.  Being a matroid is necessary but not sufficient
for this to hold.  V\'amos~\cite{Vam} suggested that there is no finite
axiom system for realizability of matroids.  Mayhew, Newman and
Whittle~\cite{MWN,MNW} finally proved this fact, and Sullivant~\cite{Sul}
established the same result for gaussoids.

One of the many axiom systems for matroids is the combinatorial
compatibility with the quadratic Pl\"ucker relations that define the
Grassmannian~\cite[\S~4]{DW91}.  Our aim is to derive the analogous
result for gaussoids. The role of the Grassmannian ${\rm Gr}(n,2n)$ is now played by a
projection of the Lagrangian Grassmannian
${\rm LGr}(n,2n)$, namely the variety~$V(J_n)$.

Let $f \in J_n$ be any polynomial relation among principal and
almost-principal minors.  A~subset $\mathcal{G}$ of $\mathcal{A}$ is
\emph{incompatible with $f$} if precisely one monomial in $f$ has no
unknown in~$\mathcal{G}$.  Otherwise $\mathcal{G}$ is {\em
compatible with} $f$.  Hilbert's Nullstellensatz implies that
$\mathcal{G}$ is compatible with all $f$ in $J_n$ if and only if it is
realizable by a symmetric $n \times n$-matrix $\Sigma$ with complex
entries.

The ideal $J_n$ contains two classes of distinguished trinomials of
degree two, namely the {\em square trinomials}
(\ref{eq:trinomial2face}), one for each $2$-cube in the $n$-cube, and
the {\em edge trinomials} (\ref{eq:trinomialedge}), one for each
$1$-cube in a $3$-cube in the $n$-cube.  The total number of these
trinomials equals
\begin{equation}
\label{eq:trinomialnumber}
  2^{n-2} \binom{n}{2}\, + \, 12 \cdot 2^{n-3} \binom{n}{3} \,\, =
  \,\,  2^{n-3} n (n-1) (2n-3) .
\end{equation}

To represent the gaussoid axioms algebraically, we use the $12 \cdot 2^{n-3} \binom{n}{3}$ edge trinomials.

\begin{example} \label{ex:J3early} Fix $n=3$.  There are $12$ edge
trinomials, one for each edge in Figure \ref{fig:eins}:
\[
  \begin{matrix}
 p_{1} a_{23} - p a_{23|1} - a_{12} a_{13}\, , \,\,\,
 p_{2} a_{13} - p a_{13|2} - a_{12} a_{23}\, , \,\,\,
 p_{3} a_{12} - p a_{12|3} - a_{23} a_{13}, \\
 p_{12} a_{13} - p_{1} a_{13|2} - a_{12} a_{23|1},\,\,
 p_{12} a_{23} - p_{2} a_{23|1} - a_{12} a_{13|2},\,\,
 p_{13} a_{12} - p_{1} a_{12|3} - a_{13} a_{23|1}, \\
 p_{13} a_{23} - p_{3} a_{23|1} - a_{13} a_{12|3}, \,\,
 p_{23} a_{12} - p_{2} a_{12|3} - a_{23} a_{13|2},\,\,
 p_{23} a_{13} - p_{3} a_{13|2} - a_{23} a_{12|3}, \\
 p_{123} a_{12} - p_{12} a_{12|3} - a_{23|1} a_{13|2},\,
p_{123} a_{13} - p_{13} a_{13|2} - a_{23|1} a_{12|3},\,
p_{123} a_{23} - p_{23} a_{23|1} - a_{12|3} a_{13|2}.
\end{matrix}
\]
The subsets $\mathcal{G}$ of $\mathcal{A}$ that are compatible with
these quadrics are precisely the sets in (\ref{eq:elevengaussoids}).
The full list of all $21$ generators of $J_3$, grouped by symmetry
class, appears in Example \ref{ex:J3}.  \hfill $\diamondsuit$
\end{example}

The edge trinomials for $n \geq 4$ are obtained by replicating these
$12$ quadrics on every $3$-face of the $n$-cube.  We can replace the
indices $1,2,3$ in the first quadric by $i,j,k$ and then add any set
$L \subseteq [n] \backslash \{i,j,k\}$ to get the trinomial
$\, p_{iL} \cdot a_{jk|L} - p_L \cdot a_{jk|iL} - a_{ij|L} \cdot
a_{ik|L}\,$ in $J_n$.

\begin{example} Fix $n=4$. The $4$-cube has $24$ two-dimensional
faces, so $\mathcal{A}$ has $2^{24} = 16777216$ subsets. Only $679$ of
these are gaussoids.  This was found in \cite[Remark~6]{LM}. The
gaussoids on $[4]$ are precisely the subsets $\mathcal{G}$ of
$\mathcal{A}$ that are compatible with the $96$ edge
trinomials. \hfill $\diamondsuit$
\end{example}

The following is our main result in Section~\ref{s:axioms}.  It
generalizes the previous two examples.

\begin{theorem}
\label{thm:gaussoid}
The following conditions are equivalent for a set
$\mathcal{G}$ of
$2$-faces of the $n$-cube:
\begin{itemize}
\item[(a)] $\mathcal{G}$ is a gaussoid, i.e.~the four axioms
  \ref{it:G1}--\ref{it:G4} are satisfied for $\mathcal{G}$;
\item[(b)] $\mathcal{G}$ is compatible with all  edge trinomials
(\ref{eq:trinomialedge}).
\end{itemize}
\end{theorem}

\begin{proof}
We begin by showing the implication from (b) to (a). For each
of the four gaussoid axioms we list either one or two
of the edge trinomials that are relevant:
\begin{itemize}
\item[\ref{it:G1}]
  $  a_{ij|L} a_{jk|iL} + a_{ik|jL} p_{iL}  -  a_{ik|L} p_{ijL} \,$ and
$\,  a_{ij|L} p_{ijkL}  - a_{ik|jL}  a_{jk|iL} -   a_{ij|kL} p_{ijL}$,  \vspace{-0.07in}
\item[\ref{it:G2}]
$ a_{ij|kL} p_{ijL} + a_{ik|jL}a_{jk|iL} - a_{ij|L}p_{ijkL}$
 and $ a_{ik|jL}p_{ikL} + a_{ij|kL}a_{jk|iL}      - a_{ik|L}p_{ijkL}$, \vspace{-0.07in}
\item[\ref{it:G3}]
$a_{ij|L}p_{kL} - a_{ik|L}a_{jk|L} - a_{ij|kL}p_L\,$ and
$a_{ij|L}a_{jk|L} - a_{ik|L}p_{jL} + a_{ik|jL}p_L$,  \vspace{-0.07in}
\item[\ref{it:G4}] $ a_{ij|L} p_{kL} - a_{ij|kL}  p_L -  a_{ik|L} a_{jk|L}$.
\end{itemize}
Compatibility with these quadrics implies the axiom. For instance,
consider axiom \ref{it:G1}.  Suppose that $a_{ij|L}$ and $a_{ik|jL}$
are in $\mathcal{G}$.  Then the first two terms of
$a_{ij|L}a_{jk|iL} + a_{ik|jL}p_{iL} - a_{ik|L}p_{ijL}$ have an
unknown in $\mathcal{G}$. Since $p_{ijL}$ cannot be an element
of~$\mathcal{G}$, we conclude that $a_{ik|L}$ is in
$\mathcal{G}$. Similarly, if the set $\mathcal{G}$ is compatible with
the edge trinomial
$a_{ij|L}p_{ijkL} - a_{ik|jL}a_{jk|iL} - a_{ij|kL}p_{ijL}$ then we can
conclude that $a_{ij|kL}$ is in $ \mathcal{G}$.  The other three
axioms are shown similarly.

For the implication from (a) to (b) we first note that the statement was already shown for $n=3$.
Namely, each of the $11$ gaussoids is compatible with the $12$ edge trinomials. Now, suppose $n \geq 4$.
 Each of the gaussoid
axioms only refers to collections of unknowns $a_{ij|K}$ that lie within a particular
$3$-face of the $n$-cube. This means that a subset $\mathcal{G}$ of $\mathcal{A}$
is a gaussoid if and only if the restriction of $\mathcal{G}$ to any $3$-face is
one of the $11$ gaussoids on $3$ symbols. The same restriction property holds for
compatibility with the edge trinomials.
\end{proof}

Among the $679$ gaussoids for $n =4$, precisely $629$ are realizable.
The other $50$ are eliminated by the higher axioms in
\cite[Lemma~2.4]{DX} and \cite[Lemma~10]{LM}.  In
Section~\ref{s:realizable} we initiate a similar analysis for
$n = 5$.  Of course, by \cite{Sul}, we cannot hope for a complete
axiom system for Gaussian realizability, and it makes sense to focus on gaussoids and
their relation to the combinatorics of quadrics in $J_n$. This
relation has a striking similarity to matroid theory. It can be
derived from the combinatorics of the the Grassmann--Pl\"{u}cker
relations.  This approach was initiated thirty years ago by Dress
and Wenzel \cite{DW91} and extended recently by Baker and Bowler \cite{BB}.
  The extent to which matroid
theory and gaussoid theory can be further developed in parallel
remains to be investigated.  It seems promising to study {\em gaussoids over hyperfields}.
Here is one concrete conjecture that points in such a direction.

\begin{conjecture} \label{conj:notjustrinomials} Every gaussoid is
compatible with all quadrics in $J_n$, not just trinomials.
\end{conjecture}

A proof for $n \leq 4$ is given in Corollary
\ref{cor:notjusttrinomials}, but that proof technique does not
generalize.  To prove Conjecture~\ref{conj:notjustrinomials} for
$n \geq 5$, it suffices to check compatibility with those quadrics
that are circuits in the subspace $(J_n)_2$ of the space of all
quadrics.  Each circuit lies in one of the weight components described
in Section~\ref{s:symmetry}.  However, that check would amount to a
prohibitive computation, even for $n=5$, because there are too many
circuits.

As support for
Conjecture~\ref{conj:notjustrinomials} we verified compatibility with
the quadrics in Theorem~\ref{thm:quadricgens} for $n=5,6$.  In
general, quadrics with two or more terms that are products of only $p$
variables, such as the square trinomials in \eqref{eq:trinomial2face},
need not be checked, as every subset of $\mathcal{A}$ is compatible
with them.  This situation changes for the valuated gaussoids of
Section~\ref{s:valuatedG}.

Minors and duality play an important role in matroid theory.  The
minors of a matroid are obtained by the iterated application of
deletions and contractions.  These two operations are reversed under
matroid duality.  For gaussoids, the roles of deletion and contraction
are played by {\em marginalization} and {\em conditioning}.  These
statistical operations are also swapped by the duality
$\Sigma \leftrightarrow \Sigma^{-1}$.
Let $\mathcal{G}$ be any gaussoid on $[n]$. The {\em dual gaussoid}
$\mathcal{G}^*$ of $\mathcal{G}$ is
\begin{equation}
\label{eq:dualgaussoid}
  \mathcal{G}^* = \bigl\{ \,a_{ij|[n] \backslash
    (K\cup\{i,j\})}\,\,:\,\, a_{ij|K} \in \mathcal{G} \,\bigr\}.
\end{equation}
For any element $ u \in [n]$, the {\em marginal gaussoid}
$\mathcal{G} \backslash u$ is the gaussoid on $[n] \backslash \{u\}$
given by
\[
  \mathcal{G} \backslash u \,\,\, = \,\,\, \,\bigl\{ \,a_{ij|K}
  \in\mathcal{G} \,\,:\,\, u \not\in \{i,j\} \cup K \,\bigr\}.
\]
Similarly, the {\em conditional gaussoid} $\,\mathcal{G}/ u \,$ is the
gaussoid on $[n] \backslash \{u\}$ given by
\[
  \mathcal{G} / u \,\,\, = \,\,\, \,\bigl\{ \,a_{ij|K} \,:\, a_{ij|K\cup\{u\}} \in \mathcal{G} \,\bigr\}.
\]
We have the following basic result relating these minors and duality:

\begin{proposition}
  If $\mathcal{G}$ is a gaussoid on $[n]$ and $u \in [n]$ then both
  $\mathcal{G} \backslash u $ and $\mathcal{G} / u $ are gaussoids.
  If $\mathcal{G}$ is realizable then so are $\mathcal{G}^*$,
  $\mathcal{G} \backslash u $, and $\mathcal{G} / u $. The following
  duality relation holds:
\begin{equation}
\label{eq:dualityrelation}
 \bigl(\mathcal{G} \backslash u \bigr)^* \, = \, \mathcal{G}^*/u \quad
 \text{ and } \quad
\bigl(\mathcal{G} / u \bigr)^* \, = \, \mathcal{G}^* \backslash u .
\end{equation}
\end{proposition}

\begin{proof}
The set of edge trinomials in $J_n$ is invariant under the duality
operation that swaps $p_K$ with $p_{[n] \backslash K}$ and also swaps
$ a_{ij|K}$ with $a_{ij|[n] \backslash (K \cup \{i,j\})}$.
Theorem~\ref{thm:gaussoid} hence ensures that $\mathcal{G}^*$ is a
gaussoid.  The duality operation preserves realizability: if a
positive definite matrix $\Sigma$ realizes $\mathcal{G}$, then its
inverse $\Sigma^{-1}$ realizes $\mathcal{G}^*$ by~\cite[Corollary~1
and Lemma~2]{LM}.

A similar argument works for marginalization and conditioning.  The
edge trinomials for $[n] \backslash \{u\}$ appear among those for
$[n]$, and similarly if we augment the index set $K$ with~$u$.  That
realizability is preserved under these operations is
\cite[Lemma~1]{Sim}.  Indeed, if $\Sigma$ realizes~$\mathcal{G}$, then
we obtain a realization of $\mathcal{G} \backslash u$ by deleting row
$u$ and column $u$ from $\Sigma$, and we obtain a realization of
$\mathcal{G} / u$ by taking the Schur complement of $\Sigma$ with
respect to $u$.

The duality relations~\eqref{eq:dualityrelation} are verified by a
direct check, bearing in mind that two of the duals in this formula
are taken with respect to the index set $[n] \backslash \{u\}$ instead
of~$[n]$.
\end{proof}

Kenyon and Pemantle~\cite{KP} initiated the study of
the ideal $J_n$ from the perspective of cluster algebras.
They conjectured a formula for the entries of $\Sigma$ in terms of
principal and almost-principal minors whose index sets are connected.
 That conjecture was proved by Sturmfels, Tsukerman, and Williams in~\cite{STW}.
As explained in~\cite[\S~5]{STW}, this is closely related to
formulas for partial correlations in statistics~\cite{Joe1}.  If
$\Sigma$ is a covariance matrix, the associated {\em correlation
matrix} has ones on the diagonal and off-diagonal entries
$\rho_{ij} = a_{ij}/\sqrt{p_i p_j}$. More generally, the {\em partial
correlations} of the Gaussian distribution  given by
$\Sigma$ are the quantities
\begin{equation}
\label{eq:parcor}
\rho_{ij|K} \,\,\, = \,\,\, \frac{a_{ij|K}}{\sqrt{ p_{iK} p_{jK}}}.
\end{equation}
Joe and his collaborators discuss the algebraic relations among the
$\rho_{ij|K}$ and construct subsets that serve as convenient
transcendence bases modulo these relations.  Their {\em d-vines} in
\cite{Joe1} correspond precisely to the {\em standard networks} of
Kenyon and Pemantle in~\cite{KP}.  Our results on gaussoids and the
ideal $J_n$ immediately imply new constraints on partial correlations.

\section{Symmetry}\label{s:symmetry}

We are interested in the ideal $J_n$ of algebraic relations among
the $2^n$ principal minors $p_L$ and the $\binom{n}{2} 2^{n-2}$
 almost-principal minors $a_{ij|K}$ of a  symmetric $n \times n$-matrix
 of unknowns.  The analogous problem for principal minors alone
 was solved (set-theoretically) by Oeding \cite{lukediss, Oed}.
He showed that the variety of  the elimination ideal $J_n \cap \RR[\mathcal{P}]$
is defined by quartics.

 \begin{example}
 Eliminating the six unknowns in $\mathcal{A}$ from $J_3$, we obtain the principal ideal
\[
  J_3 \cap \RR[\mathcal{P}] \, = \,
\langle\,   p^2 p_{123}^2 + p_1^2 p_{23}^2 + p_2^2 p_{13}^2 + p_3^2 p_{12}^2
+4 p p_{12} p_{13} p_{23} +4 p_1 p_2 p_3 p_{123}
- \cdots -2 p_1 p_3 p_{12} p_{23}  \,     \rangle.
\]
The quartic generator is the $2 \times 2 \times 2$ hyperdeterminant.
This fact was first found in~\cite{HS}.
\hfill $\diamondsuit$
\end{example}

Oeding's result is based on the representation theory of the group
$G = {\rm SL}_2(\RR)^n$.  We aim to understand $J_n$ by using this
technique.  The point of departure is the observation that $G$
acts on the space $W$ spanned by the principal and almost-principal
minors.  This action is induced by the $G$-action on the space of
$n \times 2n$-matrices.  Here, the group ${\rm SL}_2(\RR)$ in the
$i$th factor acts by replacing columns $i$ and $n+i$ by linear
combinations of these two columns.  If we apply this to
\eqref{eq:IdSigma} and then multiply by the inverse of the left
$n \times n$-block then the right $n \times n$-block is again
symmetric.  See \cite[Lemma~13]{HS} for a proof of this crucial
observation.

In this section we study the structure of the $G$-module $W$.
Let $V_i \simeq \RR^2$ denote the defining representation of
the $i$-th factor ${\rm SL}_2(\RR)$.
Let  $W_{\rm pr}$ be the space spanned by all principal minors
and $W_{\rm ap}^{ij}$ the space spanned by the almost-principal
minors $a_{ij|K}$ where $i,j$ are fixed and $K$ runs
over subsets of $[n] \backslash \{i,j\}$. The following is similar to \cite[Theorem~1.1]{Oed}:

\begin{lemma}
\label{lem:eins} We have the following isomorphisms of irreducible $G$-modules:
$$ W_{\rm pr} \,\simeq \,\bigotimes_{i=1}^n V_i \qquad \hbox{and} \qquad
W_{\rm ap}^{ij} \,\,\simeq \bigotimes_{k \in [n] \backslash \{i,j\}} V_k
\quad \hbox{for $\,\, 1 \leq i < j \leq n $.} $$
\end{lemma}

We use the unknown $x_i$ to refer to the
highest weight of the $G$-module $V_i$.
The highest weight of a tensor product of such modules is the
 product of the corresponding $x_i$.   For instance,
${\rm Sym}_2(V_1) \otimes V_2$ has highest weight $x_1^2 x_2$.
   The {\em formal character} of
a $G$-module is the sum of the Laurent monomials representing the
weights in a weight basis.  Let
$W = W_{\rm pr} \oplus \bigoplus_{i,j} W_{\rm ap}^{ij}$ be the
$G$-module of principal and almost-principal minors. The~set
$ \mathcal{A} \cup \mathcal{P} $ is a distinguished weight basis of
$W$.  By Lemma \ref{lem:eins}, the formal character of $W$ equals
\begin{equation}
\label{eq:formalcharW}
\prod_{i=1}^n(x_i+x_i^{-1})\,\,\,  + \sum_{1 \leq i < j \leq n} \prod_{k \in [n] \backslash \{i,j\}} \!\! (x_k + x_k^{-1}).
\end{equation}
Our prime ideal $J_n$ lives in the polynomial ring
$\,{\rm Sym}_*(W) = \bigoplus_{d=0}^\infty {\rm Sym}_d(W) =
\RR[\mathcal{P} \cup \mathcal{A}]$. It is invariant under the
$G$-action. The weight of a monomial in
$ \RR[\mathcal{P} \cup \mathcal{A}]$ is a vector in~$\ZZ^n$, namely,
the exponent vector of the corresponding Laurent monomial in $x_1,\ldots,x_n$.

We focus on the $G$-module of all quadrics, ${\rm Sym}_2(W)$.  Its dimension equals
\[{\rm dim}({\rm Sym}_2(W)) \,\, = \,\, \bigl( 2^n + 2^{n-2} \binom{n}{2} +1\bigr) \cdot \bigl(2^{n-1} + 2^{n-3} \binom{n}{2} \bigr). \]
The formal character of ${\rm Sym}_2(W)$ is the sum of all pairs of products
(with repetition allowed) of the $ 2^n + 2^{n-2} \binom{n}{2} $ Laurent monomials
that appear in the expansion of (\ref{eq:formalcharW}).

Each irreducible $G$-module has the form
\[S_{d_1 d_2 \cdots d_n} \,\, = \,\,\,
\bigotimes_{i=1}^n {\rm Sym}_{d_i}(V_i),
\]
where $d_1,d_2,\ldots,d_n$ are nonnegative integers.
In Oeding's work \cite{lukediss, Oed}, this module was written as $S_{d_1 0} S_{d_2 0} \cdots S_{d_n 0}$.
The formal character of the irreducible $G$-module $S_{d_1 d_2 \cdots d_n}$ equals
\begin{equation}
\label{eq:irredchar} \prod_{i=1}^n \sum_{\ell=0}^{d_i} x_i^{d_i - 2\ell} \,\, = \,\,
x_1^{d_1} x_2^{d_2} \cdots x_n^{d_n} \,+\,
\hbox{lower terms}.
\end{equation}
Our task is to express the formal character of ${\rm Sym}_2(W)$ as
a sum of Laurent polynomials (\ref{eq:irredchar}), and to
identify the submodule $(J_n)_2$ in terms of the
 irreducible $G$-modules in ${\rm Sym}_2(W)$.

\begin{example}\label{ex:J3} Let $n=3$.
The $8$ principal and $6$ almost-principal minors span the $G$-module
\[W \,\,\,  = \,\,\,  S_{111} \oplus S_{100} \oplus S_{010} \oplus S_{001}. \]
This space of quadrics has dimension $105$. It decomposes into irreducible $G$-modules as
\[{\rm Sym}_2(W) \,\,\, = \,\,\,
 S_{222} \oplus S_{211} \oplus S_{121} \oplus S_{112}
\oplus 2 S_{200} \oplus 2 S_{020} \oplus 2 S_{002} \oplus 2 S_{110} \oplus 2 S_{101} \oplus 2 S_{011}.
\]
The ring $ {\rm Sym}_*(W) = \RR[\mathcal{P} \cup \mathcal{A}] $ has
$8$ unknowns $p_I$ and $6$ unknowns $a_{ij|K}$. They are identified
with the vertices and facets of the $3$-cube
(cf.~Figure~\ref{fig:3cube}).  The weights of the $14$ unknowns~are
\[
\begin{matrix}
{\rm unknown}  & & a_{12} &  a_{12|3} & \cdots & a_{23|1} &
p & p_1 & \cdots & p_{123}
\\ {\rm weight}      & & (0,0,1) &  (0,0,-1)   & \cdots & (-1,0,0) &
(1,1,1) & (-1,1,1) & \cdots & (-1,-1,-1)
\\
\end{matrix}
\]
 The $21$-dimensional space of quadrics in $J_3$ generates the ideal. As a $G$-module,
\[(J_3)_2 \,\,\, = \,\,\, S_{200} \oplus  S_{020} \oplus  S_{002} \oplus  S_{110} \oplus  S_{101} \oplus  S_{011}. \]
We display an explicit weight basis for each summand, beginning
 with the $12$ edge trinomials:\[
  S_{110} \qquad \qquad \qquad \qquad
\begin{bmatrix}
(1, 1, 0) & &  a_{13} a_{23} + a_{12|3} p - a_{12} p_3  \\
(1,-1, 0) & &  a_{13|2} a_{23} + a_{12|3} p_2 - a_{12} p_{23}  \\
(-1,1, 0) & &  a_{13} a_{23|1} + a_{12|3} p_1 - a_{12} p_{13}  \\
\,(-1,-1,0) & &  a_{13|2} a_{23|1} + a_{12|3} p_{12} - a_{12} p_{123}\,
\end{bmatrix} \qquad
\]
\[
 S_{101} \qquad \qquad \qquad \qquad
\begin{bmatrix}
(1, 0, 1) & &  a_{12} a_{23} + a_{13|2} p - a_{13} p_2 \\
(1, 0,-1) & &  a_{12|3} a_{23} + a_{13|2} p_3 - a_{13} p_{23} \\
(-1,0, 1) & &  a_{12} a_{23|1} + a_{13|2} p_1 - a_{13} p_{12} \\
\, (-1,0,-1) & &  a_{12|3} a_{23|1}  + a_{13|2} p_{13} - a_{13} p_{123} \,
\end{bmatrix} \qquad
\]
\[
 S_{011} \qquad \qquad \qquad \qquad
\begin{bmatrix}
(0, 1, 1) & &    a_{12} a_{13} + a_{23|1} p - a_{23} p_1 \\
(0, 1,-1) & &  a_{12|3} a_{13} + a_{23|1} p_3 - a_{23} p_{13} \\
(0,-1, 1) & &  a_{12} a_{13|2} + a_{23|1} p_2  - a_{23} p_{12} \\
\, (0,-1,-1) & &  a_{12|3} a_{13|2} + a_{23|1} p_{23} - a_{23} p_{123} \,
\end{bmatrix} \qquad
\]
See Example \ref{ex:J3early}. The last three $G$-modules account for
the square trinomials:
\[
S_{200} \qquad \qquad \qquad
\begin{bmatrix}
(2, 0, 0)  & & a_{23}^2 + p p_{23} - p_{2} p_{3} \\
(0, 0, 0) & &  2 a_{23} a_{23|1} + p p_{123} + p_1 p_{23} - p_2 p_{13} -p_{12} p_3 \, \\
\,(-2,0, 0) &  & a_{23|1}^2 + p_1 p_{123} - p_{12} p_{13}
\end{bmatrix}
\]
\[
S_{020} \qquad \qquad \qquad
\begin{bmatrix}
(0, 2, 0) & & a_{13}^2 + p p_{13} - p_1 p_3 \\
(0, 0, 0) & &  2 a_{13} a_{13|2}+p p_{123} + p_2 p_{13} - p_1 p_{23} - p_{12} p_3 \,\\
\,(0,-2, 0) & &  a_{13|2}^2 + p_2 p_{123} - p_{12} p_{23}
\end{bmatrix}
\]
\[
S_{002} \qquad \qquad \qquad
\begin{bmatrix}
(0, 0, 2) & &    a_{12}^2 +        p p_{12} - p_1 p_2  \\
(0, 0, 0) & &  2 a_{12} a_{12|3} + p p_{123}+ p_3 p_{12}  - p_1 p_{23} - p_{13} p_2 \,\\
\, (0, 0,-2) & &        a_{12|3}^2 + p_3  p_{123} - p_{13} p_{23}
\end{bmatrix}
\]

The case $n=3$ is so small that every minor of $\Sigma$ is either
principal or almost-principal. Hence $J_3$ is the ideal defining the
Lagrangian Grassmannian ${\rm LGr}(3,6) \subset \PP^{13}$.
This variety has dimension $6$ and degree $16 = 6!/(1^3  3^2 5)$.
The following  code in  {\tt Macaulay2}~\cite{M2}  computes the $21$ quadrics from the $35$
quadrics that cut out the Grassmannian ${\rm Gr}(3,6) $ in $\PP^{19}$:
\begin{verbatim}
R = QQ[p123,p124,p134,p234,p125,p135,p235,p145,p245,p345,
       p126,p136,p236,p146,p246,p346,p156,p256,p356,p456];
I = Grassmannian(2,5,R) + ideal(p124-p236,p125+p136,p134+p235,
                                p346+p245,p356-p145,p256+p146);
J3 = eliminate({p124,p134,p125,p356,p256,p346},I)
\end{verbatim}
From the free resolution (computed with {\tt res J3}) it can be verified that $J_3$ is Gorenstein.

Each of the $20$ generators of the polynomial ring {\tt R} equals (up to sign)
one of the $14$ variables in $\mathcal{P} \cup \mathcal{A}$.
The precise identification is given by the following ordered list of length~$20$:
\begin{equation}
\label{eq:identify}
\begin{matrix}
 p & a_{13} & -a_{12} & p_1 & a_{23} & -p_2 & a_{12}  & a_{13|2} &  -a_{23|1} & p_{12} \\
  p_3 & -a_{23} &  a_{13} & a_{12|3} & -p_{13} & a_{23|1} & p_{23} & -a_{12|3} & a_{13|2} & p_{123}
  \end{matrix}
\end{equation}

One comment for algebraic geometers: canonical curves of genus $9$ are
obtained by intersecting $V(J_3) = {\rm LGr}(3,6)$ with subspaces
$\PP^8$ in $\PP^{13}$.  This was shown by Mukai and further developed
by Iliev and Ranestad~\cite{IR}, who derive the $21$ quadrics
explicitly in \cite[\S~2.3]{IR}.  \hfill $\diamondsuit$
\end{example}

\begin{example}\label{ex:J4} Let $n{=}4$.
There are $16$ principal and $24$ almost-principal minors. They span
\[ W \,\,\, = \,\,\,  S_{1111} \oplus S_{1100} \oplus S_{1010} \oplus S_{1001} \oplus S_{0110}
\oplus S_{0101} \oplus S_{0011}.  \]
The space of quadrics has dimension $820$. It decomposes into irreducible $G$-modules as
\[
\begin{matrix} {\rm Sym}_2(W) & = &
 S_{2222} \oplus S_{2211} \oplus S_{2121} \oplus S_{2112} \oplus S_{1221} \oplus
 S_{1212} \oplus S_{1122} \oplus 2 S_{2200} \oplus 2 S_{2020} \\ & & \oplus\, 2 S_{2002}  \oplus \,2 S_{0220}
 \oplus 2 S_{0202} \oplus 2 S_{0022} \oplus 2 S_{2110} \oplus 2 S_{2101} \oplus 2 S_{2011}
 \oplus 2 S_{1210} \\ & & \oplus\, 2 S_{1201} \oplus 2 S_{0211}  \oplus\, 2 S_{1120} \oplus
 2 S_{1021} \oplus 2 S_{0121} \oplus 2 S_{1102} \oplus 2 S_{1012} \oplus
 2 S_{0112} \\  & & \oplus \, 3 S_{1111} \oplus 3 S_{1100} \oplus\, 3 S_{1010} \oplus 3 S_{1001} \oplus 3 S_{0110} \oplus
 3 S_{0101} \oplus 3 S_{0011} \oplus 7 S_{0000}.
 \end{matrix}
\]
The $226$-dimensional submodule of quadrics that vanishes on our variety equals
\begin{equation}
\label{eq:fourtwo}
\begin{matrix} (J_4)_2 &=&
S_{2200} \oplus S_{2020} \oplus S_{2002} \oplus S_{0220} \oplus S_{0202} \oplus
 S_{0022} \oplus S_{2110} \oplus S_{2101} \oplus S_{2011}   \\ & &
 \oplus \, S_{1210} \oplus S_{1201}  \oplus S_{0211} \oplus
 S_{1120} \oplus S_{1021} \oplus S_{0121} \oplus S_{1102} \oplus
 S_{1012}   \\  & &  \oplus \,
 S_{0112} \oplus S_{1100} \oplus S_{1010} \oplus S_{1001}
 \oplus S_{0110} \oplus S_{0101} \oplus S_{0011} \oplus 4 S_{0000}.
 \end{matrix}
\end{equation}
Each of the four copies of the one-dimensional module $S_{0000}$ is
spanned by a $G$-invariant quadric in the ideal $J_4$.  Here is one
such invariant that involves none of the principal minors:
\[
a_{14}a_{14|23}-a_{14|2} a_{14|3}-a_{23}a_{23|14}+a_{23|1} a_{23|4}.
\]
This quadric can be derived from the quadrics in
Theorem~\ref{thm:quadricgens} \eqref{q4}.  It is instructive to locate
the $24$ square trinomials and the $96$ edge trinomials inside the summands seen in
(\ref{eq:fourtwo}).  \hfill $\diamondsuit$
\end{example}

In Section~\ref{s:quadrics} we study the quadrics in~$J_n$. This uses
the action of the Lie algebra $\mathfrak{g}$  of the group $G$.  The situation
differs from that in \cite{lukediss, Oed}. Oeding's hyperdeterminantal ideal is generated
by a single irreducible module for the action of $G \rtimes S_n$ on
${\rm Sym}_4(W_{\rm pr})$. In our case, there are many irreducibles
even modulo the action of $S_n$.  The space $(J_3)_2$ in
Example~\ref{ex:J3} decomposes into two irreducible
$G \rtimes S_3$-modules, and $(J_4)_2$ in Example~\ref{ex:J4}
decomposes into five irreducible $G \rtimes S_4$-modules.  This
complexity accounts for the difficulties in
Section~\ref{s:quadrics}.

We now shift gears and discuss a collection of finite groups that
act on the combinatorial structures studied in this paper.  These
finite groups arise from the following inclusions:
\begin{equation}
\label{eq:fourgroups}
S_n \,\,\subset \,\,  (\ZZ/2 \ZZ)^n \rtimes S_n   \,\, \subset \,\,
(\ZZ/4 \ZZ)^n  \rtimes S_n  \,\, \subset \, \, G \rtimes S_n .
\end{equation}
The symmetric group $S_n$ acts by permuting indices in the unknowns
$p_I$ and $a_{ij|K}$, and by simultaneously permuting the rows and
columns of the matrix $\Sigma$ in~\eqref{eq:IdSigma}.  The third group
in (\ref{eq:fourgroups}) is obtained by taking the following cyclic
subgroup in each factor ${\rm SL}_2(\RR)$:
\[
\ZZ/4 \ZZ  \quad \simeq \quad \biggl\{
\begin{pmatrix} 1 & 0 \\ 0 & 1 \end{pmatrix}, \,
\begin{pmatrix} 0 & 1 \\ -1 & 0 \end{pmatrix}, \,
\begin{pmatrix} -1 & 0 \\ 0 & -1 \end{pmatrix}, \,
\begin{pmatrix} 0 & -1 \\ 1 & 0 \end{pmatrix} \biggr\}.
\]

\begin{remark} \label{rem:notpreserved} The action of $(\ZZ/4 \ZZ)^n$
on the Lagrangian Grassmannian ${\rm LGr}(n,2n)$ takes symmetric
matrices to symmetric matrices, but it changes their signatures. It thus
does not preserve the property that $\Sigma$ is positive definite.  In
fact, already the induced action by the hyperoctahedral group
(discussed below) does not preserve realizability of gaussoids.
\end{remark}

Consider the subgroup $R_n \,= \,\bigl\{(\pm {\rm Id}_2, \pm {\rm Id}_2, \ldots,
\pm {\rm Id}_2) \bigr\} \,$ of $\,{\rm SL}_2(\RR)^n$.
 Each element in this group corresponds to
an $n \times n$-diagonal matrix $D$ with entries in $\{-1,+1\}$.  {\em
Reorientation} is the action of $R_n$ that maps $\Sigma$ to
$D \Sigma D$.  This does not change the principal minors of $\Sigma$.  In
particular, if $\Sigma$ is positive definite, then so is $D \Sigma D$.
Under this action, the almost-principal minor $a_{ij|K}$ transforms
into $D_{ii} D_{jj} a_{ij|K} = \pm a_{ij|K}$.  This action is trivial
for gaussoids, but it is non-trivial for oriented gaussoids, as we shall see in
Section~\ref{s:orientedAndPos}.

In order to get a faithful action on the set of gaussoids we need to
take the quotient of $ (\ZZ/4 \ZZ)^n \rtimes S_n$ modulo its normal
subgroup $(\ZZ/2 \ZZ)^n \rtimes \{{\rm Id} \}$. The resulting group
$(\ZZ/2 \ZZ)^n \rtimes S_n$ is the {\em hyperoctahedral group~$B_n$}.
It acts on the set of gaussoids as the symmetry group of the $n$-cube.
The quotient $(\ZZ/2\ZZ)^n$ acts by swapping indices in and out from
the index sets $I$ and $K$ in the quantities $p_I$ and $a_{ij|K}$.
When expressed in terms of $\Sigma$, the latter action looks like a
fusion of matrix inversion and Schur complements.  Consider the
subgroup of the hyperoctahedral group given by
$\ZZ/2\ZZ =\{({\rm Id}_2, \ldots,{\rm Id}_2), ((\begin{smallmatrix}
0 & 1\\
-1 & 0
\end{smallmatrix}), \ldots
,(\begin{smallmatrix}
0 & 1\\
-1 & 0
\end{smallmatrix}))
\}$ inside
$ (\ZZ/2 \ZZ)^n$:
\begin{equation}
\label{eq:twogroups}
 S_n \,\, \subset \,\,   (\ZZ/2 \ZZ) \rtimes S_n \,\,\subset \,\,
   (\ZZ/2 \ZZ)^n \rtimes S_n  .
\end{equation}
The group $\ZZ/2\ZZ$ in the middle of (\ref{eq:twogroups}) acts on the
set of gaussoids by the {\em duality} in (\ref{eq:dualgaussoid}).
Algebraically, this is the involution on ${\rm LGr}(n,2n)$ that maps
$\Sigma$ to its negative inverse $-\Sigma^{-1}$.  In summary, these
finite group actions are subtle.  In particular, the distinction
between the reorientation group $R_n$ and the $(\ZZ/2\ZZ)^n$-factor
of~$B_n$ is crucial.

\section{Census of Small Gaussoids}\label{s:census}

In this section we derive and discuss the following result.  The proof
for $n=5$ is by computation. It rests on using state-of-the-art
software from the field of {\em SAT solvers} \cite{sharpSat, TodaAllSAT}.

\begin{theorem} \label{thm:gaussoidcensus}
For  $n=3,4,5$, the number of gaussoids $\mathcal{G}$ is as follows,
up to symmetries:
{\rm
\begin{center}
\begin{tabular}{c||ccccccc|}
  $n$&& \hbox{all gaussoids}  & \hbox{orbits for} $S_n$ &  $\ZZ/2\ZZ \rtimes S_n$  & $(\ZZ/2\ZZ)^n \rtimes S_n$\\
  \hline
  3&& 11		      & 5	     & 4	& 4   \\
  4&& 679		      & 58	     & 42	& 19  \\
  5&& 60,212,776	      & 508,817      &   254,826	 & 16,981  \\
\end{tabular}
\end{center}
}
\end{theorem}

The second, third, and fourth column report the number of orbits under
the group actions described in (\ref{eq:twogroups}).
Theorem~\ref{thm:gaussoidcensus} for $n=3$ is
Example~\ref{ex:elevengaussoids}, where the $11$ gaussoids are
listed. There are five orbits under permuting the indices $1,2,3$. The
two singleton orbits fuse to a single orbit under the group
$\ZZ/2\ZZ \rtimes S_3$.  For instance, the gaussoids $\{a_{12}\}$ and
$\{a_{12|3} \}$ are swapped under duality. The same four orbits
persist under the action of the hyperoctahedral group
$ (\ZZ/2\ZZ)^3 \rtimes S_3$, since $|\mathcal{G}|$ is an invariant of
that action.  For $n=4$, Ln\v{e}ni\v{c}ka and Mat\'u\v{s}~\cite{LM}
showed that there are $679$ gaussoids, of which $629$ are realizable.
Their computation was confirmed by Drton and Xiao~\cite{DX}.  The
action by the hyperoctahedral group $ (\ZZ/2\ZZ)^4 \rtimes S_4$ was
not used in \cite{DX, LM}, but we find this to be natural in our
setting.

\begin{lemma}
The $679$ gaussoids for $n=4$ are organized into orbits as follows:
\begin{itemize}
\item The symmetric group $ S_4 $ of order $24$ acts on the gaussoids
by permuting indices.  There are $58$ orbits under that action. Five
of these orbits consist of non-realizable gaussoids.
\item The twisted symmetric group $\ZZ/2\ZZ \rtimes S_4$ of order $48$
acts on the gaussoids by duality and permuting indices.  This action
preserves realizability, and it has $42$ orbits.  Five of these orbits
consist of non-realizable gaussoids.
\item Under the action of the hyperoctahedral group
$(\ZZ/2\ZZ)^4 \rtimes S_4$ of order $384$, there are $19$ orbits.
Three of the orbits contain non-realizable gaussoids.
\end{itemize}
\end{lemma}

The difference between the three group actions can already be seen
for the $24$ singleton gaussoids. These correspond to the $2$-faces
of the $4$-cube. The symmetric group
$S_4$ has three distinct orbits on the set $\mathcal{A}$: the six $1 \times 1$-minors $a_{ij}$,
the twelve $2 \times 2$-minors $a_{ij|k}$, and the
six $3 \times 3$-minors $a_{ij|kl}$.
The $1 \times 1$-minors and the $3 \times 3$-minors are swapped under duality.
So, there are two orbits of size $12$ for the group $\ZZ/2\ZZ \rtimes S_4$.
Finally, the full symmetry group of the $4$-cube acts transitively on the
$2$-faces. Hence that group has only one orbit of size $24$.

The following $19$ items are the orbits of the
hyperoctahedral group $(\ZZ/2\ZZ)^4 \rtimes S_4$.
The symbol $\,\ell_m\,$ at the beginning indicates that the orbit
consists of $m$ gaussoids $\mathcal{G}$, each of cardinality $|\mathcal{G}| = \ell$.
This is followed by a list of $\ZZ/2\ZZ \rtimes S_4$-orbits,
each given  by its lexicographically first
representative. For instance, the fourth item, marked $2_{48}$, is
a hyperoctahedral orbit of size $48$
that consists of two-element gaussoids $\mathcal{G}$.
If we restrict to  permuting indices and duality then this orbit
breaks into four smaller orbits, of cardinalities $6,6,12$ and~$24$.

Five of the small orbits are distinguished by double-brackets
$[[ \,\, ]] $ instead of single curly brackets $\{ \,\, \}$.  The
$50 = 8 + 6+ 6 + 6+ 24$ elements in these five
$\ZZ/2\ZZ \rtimes S_4$-orbits are the non-realizable gaussoids.  For
instance, the eight triple gaussoids in the orbit
$[[ a_{12|3}, a_{13|4}, a_{14|2} ]]_8$ are non-realizable.  We discuss
the issue of realizability in more detail after our list.
\begin{align*}
  0_1:& \ \ \emptyset_1 \\
  1_{24}:& \ \ \{a_{12}\}_{12} \qquad \{ a_{12|3} \}_{12} \\
  2_{12}:& \ \ \{a_{12}, a_{12|34} \}_6 \qquad \{a_{12|3}, a_{12|4} \}_6  \\
  2_{48}:& \ \ 
\{a_{12}, a_{34} \}_6 \ \ \ \
\{a_{12}, a_{34|12} \}_6 \  \ \ \ 
\{ a_{12|3}, a_{34|1} \}_{12} \ \  \ \ 
\{ a_{12}, a_{34|1} \}_{24} \\
 2_{96}:& \ \ 
\{ a_{12}, a_{13|24} \}_{24} \ \ \ \
\{ a_{12}, a_{13|4} \}_{48} \ \ \ \ 
\{ a_{12|3}, a_{13|4} \}_{24} \\
 3_{32} :& \ \ 
\{ a_{12}, a_{13|24}, a_{14|3} \}_{24} \ \ \ \
[[ a_{12|3}, a_{13|4}, a_{14|2} ]]_8 \\
3_{48}:& \ \ 
\{a_{12}, a_{12|34}, a_{34}\}_{12} \  \ \  \
\{a_{12}, a_{12|34}, a_{34|1} \}_{12} \ \  \  \
\{a_{12}, a_{34|1}, a_{34|2}\}_{12} \  \ \ \
\{a_{12|3}, a_{12|4}, a_{34|1} \}_{12} \\
  3_{48}:& \ \ 
\{a_{12}, a_{12|34}, a_{34} \}_{12} \ \  \ \
\{a_{12}, a_{12|34}, a_{34|1} \}_{12} \ \  \ \
\{a_{12}, a_{34|1}, a_{34|2} \}_{12} \ \  \ \
\{a_{12|3}, a_{12|4}, a_{34|1} \}_{12} \\
  3_{192}:& \ \ 
\{a_{12}, a_{13|24}, a_{24|13}\}_{24} \ \  \
\{a_{12}, a_{13|4}, a_{24|3} \}_{24} \  \ \
	    \{a_{12}, a_{13|4}, a_{34|12} \}_{24}  \\
  & \ \  \{a_{12|3}, a_{13|4}, a_{24|1} \}_{24} \ \  \
 \{a_{12}, a_{13|24}, a_{24|3} \}_{48} \ \   \
\{a_{12}, a_{13|4}, a_{34|2} \}_{48} \\
 4_{12}:& \ \ 
\{ a_{12}, a_{12|34}, a_{34}, a_{34|12} \}_3 \qquad
\{a_{12|3}, a_{12|4}, a_{34|1}, a_{34|2} \}_3 \qquad
[[ a_{12}, a_{12|34}, a_{34|1}, a_{34|2}]]_6 \\
 4_{24}:& \ \  \{a_{12}, a_{12|3}, a_{13}, a_{13|2}\}_{24} \\
   4_{48}:& \ \ 
\{ a_{12}, a_{13|4}, a_{24|3}, a_{34|12} \}_{12} \  \ \
[[ a_{12}, a_{13|24}, a_{24|13}, a_{34}]]_6 \\
& \ \ 
[[ a_{12|3}, a_{13|4}, a_{24|1}, a_{34|2}]]_6 \  \ \
[[ a_{12}, a_{13|24}, a_{24|3}, a_{34|1}]]_{24} \\
 5_{48}:& \ \ 
\{ a_{12}, a_{12|3}, a_{13}, a_{13|2}, a_{23|14}\}_{24} \ \qquad
\{ a_{12}, a_{12|3}, a_{13}, a_{13|2}, a_{23|4}\}_{24} \\
 6_8:& \ \  \{ a_{12}, a_{12|3}, a_{13}, a_{13|2}, a_{23}, a_{23|1} \}_8 \\
 7_{48}:& \ \ 
\{ a_{12}, a_{12|3}, a_{12|34}, a_{13}, a_{13|2}, a_{24|13}, a_{24|3}\}_{24} \qquad
\{a_{12}, a_{12|3}, a_{12|4}, a_{13}, a_{13|2}, a_{24}, a_{24|1}\}_{24} \\
 12_4:& \ \ 
\{a_{12}, a_{12|3}, a_{12|34}, a_{12|4}, a_{13}, a_{13|2}, a_{13|24}, a_{13|4}, a_{14}, a_{14|2}, a_{14|23}, a_{14|3}\}_4 \\
 14_{24}:& \ \ 
\bigl(\mathcal{A} \backslash \{a_{23|14}, a_{23|4}, a_{24}, a_{24|1}, a_{24|13}, a_{24|3}, a_{34}, a_{34|1}, a_{34|12}, a_{34|2}\}\bigr)_{24} \\
 16_3:& \ \ 
\bigl(\mathcal{A} \backslash \{
a_{14}, a_{14|2}, a_{14|23}, a_{14|3}, a_{23}, a_{23|1}, a_{23|14}, a_{23|4}\}\bigr)_3 \\
  20_6:& \ \ 
\bigl(\mathcal{A} \backslash \{a_{34}, a_{34|1}, a_{34|12}, a_{34|2}\}\bigr)_6 \\
 24_1:& \ \ \bigl(\mathcal{A}\bigr)_1
\end{align*}

It is instructive to look at the list above through the lens of
Remark~\ref{rem:notpreserved}.  The action of
$(\ZZ/2\ZZ)^4 \rtimes S_4$ on the variety $V(J_4)$ and on the $679$
gaussoids can be understood via the Lagrangian Grassmannian
${\rm LGr}(4,8) \subset \PP^{39}$.  Here a symmetric
$4 \times 4$-matrix $\Sigma$ corresponds to the $4 \times 8$-matrix
$\bigl({\rm Id}_4 \,\, \Sigma\bigr)$, where ${\rm Id}_4$ is the
$4 \times 4$ identity matrix. The group $S_4$ acts on this
$4 \times 8$-matrix by simultaneously permuting rows and columns of
$\Sigma$ and of~${\rm Id}_4$.  Each factor of $(\ZZ/2\ZZ)^4$ switches
a column of ${\rm Id}_4$ with the corresponding column of $\Sigma$ and
changes the sign of one of the columns.  If one multiplies the
$4 \times 8$-matrix on the left by the inverse of its left
$4 \times 4$-block, then the result is a matrix
$\bigl({\rm Id}_4 \,\, \Sigma'\bigr)$, where $\Sigma'$ is symmetric
by~\cite[Lemma~13]{HS}.

After swapping one column and switching the
sign, the signatures of the symmetric matrices
$\Sigma'$ and $\Sigma $ differ by one. Thus, if $\Sigma$ is
positive definite, then $\Sigma'$ is not positive~definite. After
having performed this operation for all four columns,
the resulting matrix $\Sigma'$ is negative definite.
We then replace $\Sigma'$ by its negative $-\Sigma'$ to get a
positive definite matrix.  Including this last step, this group
action represents gaussoid duality, which retains realizability.

Because of this change in signature, the action of
$(\ZZ/2\ZZ)^4 \rtimes S_4$ on gaussoids
does not preserve realizability in the Gaussian sense where
all $p_I$ are to remain positive. However it does retain
a weaker notion of realizability which only requires
that the $p_I$ remain nonzero.

\begin{remark} \label{rem:RealNull}
The non-realizability of the five $\ZZ/2\ZZ \rtimes S_4$-orbits that
were highlighted above can be certified by polynomials in the ideal $J_4$.
The existence of  such certificates is guaranteed by the {\em Real Nullstellensatz}.
For example, to show that the gaussoid
$\mathcal{G} = \{a_{12|3}, a_{13|4}, a_{14|2}\}$ has no
Gaussian realization, we can use the following algebraic relation which lies in $J_4$:
\small
\[
a_{14} \bigl(a_{34}^2 p_2 p_4 p_{23} +a_{23}^2 a_{34}^2 p_{24}+p_2^2 p_3 p_4 p_{34} \bigr)
\, -\,(a_{23} a_{24} a_{34}+p_2 p_3 p_4) ( a_{24} p_4 \underline{a_{12|3}}\,
+\,a_{24} a_{23} \underline{a_{13|4}}\,+\,p_3 p_4 \underline{a_{14|2}}).
\]
\normalsize
Indeed, in any realization the second summand is zero.  However, the
first summand is nonzero because the three terms in the left
parenthesis are strictly positive.  Starting from the proofs in
\cite[Corollary~4]{LM}, we can derive similar certificates for the
other four gaussoids.
\end{remark}

\begin{corollary} \label{cor:notjusttrinomials}
Conjecture \ref{conj:notjustrinomials} is true for $n=4$.
\end{corollary}

\begin{proof}
Each of the five non-realizable gaussoids $\mathcal{G}$ has a
realizable gaussoid in its orbit under the group~$B_4$.  Hence
$\mathcal{G}$ admits a realization where all $p_I$ are nonzero, but
some are negative.  The existence of such a non-Gaussian realization
shows that $\mathcal{G}$ is compatible with every polynomial
in~$J_4$. In particular, $\mathcal{G}$ is compatible with every
quadric in $J_4$.
\end{proof}

We now come to the census of gaussoids for $n=5$.  This is the
main result in Theorem~\ref{thm:gaussoidcensus}.  It is derived by
direct computation using SAT solvers.  Here SAT stands for the
satisfiability problem of propositional logic. This problem is NP-complete.
However, there have been considerable advances in solving
SAT problems in practice.  We believe that these techniques can be
useful for a wide range of problems in applied algebraic geometry.

All SAT solvers use the same input: a Boolean formula in conjunctive
normal form (cnf). A {\em cnf formula} is a conjunction of clauses,
where a clause is a disjunction of variables or negated variables.
Every Boolean formula can be brought into cnf.  A standard file format
is the \verb|DIMACS| cnf file format.  There are three natural problems for
a given cnf formula:

\textbf{SAT}: Is the formula satisfiable?

\textbf{\#SAT}: How many satisfying assignments are there?

\textbf{AllSAT}: Enumerate all satisfying assignments.

\noindent
The three problems are listed by increasing difficulty.  The third is
the most relevant for us.  For example, gaussoid enumeration is an
instance of {\bf AllSAT}.  To show this, we introduce one Boolean
variable for each element $a_{ij|K}$ of~$\mathcal{A}$.  The gaussoid
$\mathcal{G}$ consists of those variables that take the value zero.
This convention is consistent with the assignment of zero to the
variables in the gaussoid, when checking compatibility with the edge
trinomials.  The gaussoid axioms can be formulated as Boolean
formulas.  Specifically, \ref{it:G1}-\ref{it:G3} can be written as
$A \land B \Rightarrow C \land D$ where $A,B,C,D$ are statements of
the form $a_{ij|K}\in \mathcal{G}$, or $a_{ij|K} = 0$.
The implication above can be brought into cnf with two disjunctions as
follows:
\begin{gather*}
  \{A \land B \Rightarrow C \land D \} \quad \Longleftrightarrow \quad  (C \lor \neg A \lor \neg B)
\,  \land\, (D \lor \neg A \lor \neg B).
\end{gather*}
The weak transitivity axiom \ref{it:G4} is of the form $A \land B \Rightarrow C \lor D$. It has
the simple cnf
\[
  C \lor D \lor \neg A \lor \neg B.
\]
These axioms in cnf
need to be specified for all possible choices of $i,j,k,L$ in \ref{it:G1}-\ref{it:G4}.

\begin{lemma}\label{lem:GaussoidEnumWithSAT}
The enumeration of all gaussoids on $[n]$ is an instance of
the {\tt AllSAT} problem,
where the  Boolean formula in conjunctive normal form
 (cnf) has $\,7\binom{n}{3}2^{n-3}3!\,$ clauses.
\end{lemma}

\begin{proof}
  For each choice of an ordered triple $(i,j,k)$ from $ [n]$
  and a subset  $L$ of $ [n] \backslash\{i,j,k\}$, we
  have introduced seven clauses: two for each
  axiom \ref{it:G1}-\ref{it:G3} and one for \ref{it:G4}.
\end{proof}

\begin{proof}[Proof of Theorem \ref{thm:gaussoidcensus}]
For $n \leq 4$ see the discussion above.  The proof for $n=5$ consists of the following
computation.  Using Lemma~\ref{lem:GaussoidEnumWithSAT}, we expressed
the gaussoid axioms as an {\bf AllSAT} instance with $1680$ clauses.
The formula was then solved using the solver \verb|bdd_minisat_all|
due to Toda and Takehide~\cite{TodaAllSAT}.  The output is the list of
all $60212776$ gaussoids.  This count was verified independently using
Thurley's \textbf{\#SAT} solver {\tt sharpSAT}~\cite{sharpSat} on the same input.
To group the gaussoids into orbits under the actions of the three finite
groups in \eqref{eq:twogroups} we wrote our own code in {\tt
sage}~\cite{sage}.  The numbers of orbits we found are those in the
table.  Our homepage \href{https://www.gaussoids.de}{www.gaussoids.de}
contains this data and the code to reproduce it.
\end{proof}

\section{Oriented Gaussoids and Positivity}\label{s:orientedAndPos}

Theorem~\ref{thm:gaussoid} establishes a strong parallel between
matroids and gaussoids. An important feature of matroid theory is its
numerous extensions, notably~to oriented matroids~\cite{OM},
positroids~\cite{ARW}, and valuated matroids~\cite{DW}.  The analogues
in our setting are oriented gaussoids, positive gaussoids, and
valuated gaussoids.  This section is devoted to the first two of these.

Given any gaussoid $\mathcal{G}\subset\mathcal{A}$, we can assign
orientations $+$ or $-$ to the unknowns $a_{ij|K}$
in~$\mathcal{A}\backslash\mathcal{G}$.  These represent inequalities
$a_{ij|K} > 0$ and $a_{ij|K} < 0$ that express the sign of the partial
correlation (\ref{eq:parcor}) among the random variables $X_i$ and
$X_j$ given $X_K$.  Not all assignments are compatible with the edge
relations, which is a necessary condition for representability.

An \emph{oriented gaussoid} is a map $\mathcal{A} \to \{0,\pm 1\}$
with the following property for each edge trinomial: after setting
elements in $\mathcal{P}$ to $+1$ and elements in $\mathcal{A}$ to
their images, the set of resulting terms is either $\{0\}$ or
$\{-1,+1\}$ or $ \{-1,0,+1\}$.  A \emph{positive gaussoid} is an
assignment $\mathcal{A} \to \{0, 1\}$ satisfying the same compatibility
requirement. For any oriented gaussoid, the inverse image of $0$ is a
gaussoid~$\mathcal{G}$.  This is analogous to the {\em chirotope
axiom} for oriented matroids~\cite[\S~1.2]{OM}, which expresses
compatibility with the Grassmann--Pl\"ucker relations.  An oriented
gaussoid with $\mathcal{G} = \emptyset$, so that
$\mathcal{A} \rightarrow \{\pm 1\}$, is called a {\em uniform oriented gaussoid}.

Positroids are oriented matroids all of whose bases are positively
oriented.  They play an important role in representation theory and
algebraic combinatorics, and they have desirable geometric properties.
Ardila, Rinc\'on, and Williams proved a longstanding conjecture of
Da~Silva by showing that all positroids are realizable~\cite{ARW}.  In
Theorem~\ref{thm:posGaussoids} we prove the same fact for gaussoids.
Positive gaussoids are important for statistics, because they
correspond to the ${\rm MTP}_2$ distributions, which have received a
lot of attention in the recent literature~\cite{fallat,LUZ}.

We begin by discussing the enumeration of oriented gaussoids.
We start with an ordinary gaussoid $\mathcal{G}$.
The aim is to list all of its orientations.  According to
Theorem~\ref{thm:gaussoid}, when setting $\mathcal{P}$ to $1$ and
$\mathcal{G}$ to $0$ in the edge trinomials,
each trinomial either vanishes, stays a trinomial, or
becomes a binomial.  The resulting nonzero polynomials are the
\emph{mutilated edge relations}.  They combinatorially constrain the
possible orientations. Here is a simple example:

\begin{example}\label{ex:notPositive}
Fix $n=4$ and consider the singleton gaussoid
$\mathcal{G} = \{a_{34|2}\}$.  The edge tri\-nomial
$p_{12}a_{34|2}-p_{2}a_{34|12}-a_{13|2}a_{14|2}$ is mutilated to
$-a_{34|12}-a_{13|2}a_{14|2}$.  This bi\-nomial precludes four of the
eight possible assignments of signs to its three unknowns.  In
particular, assigning all $+$ is forbidden. Hence $\mathcal{G}$ is not
positively orientable. Still, $\mathcal{G}$ has $576$
orientations.\hfill$\diamondsuit$
\end{example}

Enumerating all orientations of a gaussoid $\mathcal{G}$ can be
formulated as an \textbf{AllSAT} instance.  We use one binary variable
$V_a$ for each element $a \in \mathcal{A}\backslash\mathcal{G}$.  We
set $V_a = 1$ when $a \mapsto -1$ and $V_a= 0$ when $a \mapsto +1$.
With this convention, the addition $V_a \oplus V_b$ in the field
$\mathbb{F}_2$ gives the sign of the product $ab$.  Consider a
non-mutilated edge trinomial $a - b - cd$, where $a, b, c, d$ are
elements in $\mathcal{A}$.  Compatibility means: whenever one term is
positive, another term must be negative, and vice versa. This
translates into the following Boolean formula:
\begin{equation}
  (\neg V_a \vee V_b \vee (V_c\oplus V_d))\,\, \Leftrightarrow \,\,
  (V_a \vee \neg V_b \vee \neg (V_c\oplus V_d)). \label{eq:CompatTrinom}
\end{equation}
The formula~(\ref{eq:CompatTrinom}) has a fairly short conjunctive normal form (cnf):
\[
   (V_a \vee \neg V_b \vee V_c \vee \neg V_d)
  \wedge  (V_a \vee \neg V_b \vee \neg V_c \vee V_d)
  \wedge  (\neg V_a \vee V_b \vee V_c \vee V_d)
  \wedge  (\neg V_a \vee V_b \vee \neg V_c \vee \neg V_d).
\]
If the trinomial is mutilated, then we omit from~\eqref{eq:CompatTrinom} all variables which appear no longer.

\begin{example}
Consider the cnf above for the empty gaussoid
$\mathcal{G} = \emptyset $ with $n=4$.  Applying a \textbf{\#SAT}
solver yields the number $ 5376$ for the uniform oriented gaussoids on
$n=4$. \hfill $\diamondsuit$
\end{example}

Here is our main result  on the classification of small oriented gaussoids.

\begin{theorem} \label{thm:oriegaussoidcensus}
For  $n=3,4,5$, the numbers of oriented gaussoids are as follows:
{\rm
\begin{center}
\begin{tabular}{c||ccccccc|}
 $n$ & {\rm ordinary} & {\rm oriented} & {\rm positive} & {\rm uniform} \\
 \hline
 3 & 11 & 51 & 8 & 20 \\
 4 & 679 & 34,873 & 64 & 5,376 \\
 5 & 60,212,776 & 54,936,241,913 & 1,024 & 878,349,984
\end{tabular}
\end{center}
}
\end{theorem}

\begin{proof}
Our count of oriented gaussoids is the result of a \textbf{\#SAT} computation.
Each variable in $\mathcal{A}$ can assume a value in $\{0, \pm 1\}$.
We modeled one such ternary variable with two Boolean variables
$V_a^1, V_a^2$ and a surjection
$\eta: \mathbb{F}_2^2 \to \{0, \pm 1\}$, but forbade one configuration
of $(V_a^1, V_a^2)$ so that $\eta$ becomes a bijection on all allowed
configurations.  Formula~\eqref{eq:CompatTrinom} can be adapted to
describe all oriented gaussoids. The results are summarized
in the table.
\end{proof}

The symmetries of oriented gaussoids differ in two ways from the
symmetries of gaussoids. On the one hand, there are fewer symmetries
coming from the groups in (\ref{eq:twogroups}). The two groups on the right
change the signs of the principal minors of $\Sigma$, so their action on gaussoids
does not lift to oriented gaussoids. Only the action by the
permutation group $S_n$ survives.

 On the other hand, certain new symmetries arise,  namely those given
 by {\em reorientations}. We discussed this point after Remark~\ref{rem:notpreserved}.
 They are analogous to reorientations of oriented matroids \cite[\S~1.2]{OM}.  Reorientations
act on the signs of the almost-principal minors $a_{ij|K}$ as
follows.  If $\phi : \mathcal{A} \rightarrow \{0,\pm 1\}$ is an
oriented gaussoid, and $L$ is any subset of $[n]$, then the {\em
reorientation} of $\,\phi\,$ {\em along} $\,L\,$ is the oriented
gaussoid $\phi_L : \mathcal{A} \rightarrow \{0,\pm 1\}$ given by
$\phi_L(a_{ij|K} ) = (-1)^{|\{i,j\} \cap L|} \cdot \phi(a_{ij|K})$.
There are only $2^{n-1}$ reorientations since $\phi_L =  \phi_{[n] \backslash L}$.
The {\em symmetry classes of oriented gaussoids} are the orbits of
oriented gaussoids under the semidirect product $R_n\rtimes S_n$ 
of the reorientation group $R_n$ and the symmetric group~$S_n$.

\begin{example} \label{ex:ogclass3} Let $n=3$ and consider the $S_3$-orbit of gaussoids
$\bigl\{ \{a_{12}\}, \{a_{13}\}, \{a_{23}\} \bigr\}$.
 Each gaussoid $\{a_{ij}\}$ is the support of four oriented gaussoids that are related by
reorientation. Altogether, this results in a symmetry class of size
$12 = 3 \times 4$.  We display each of these $12$ oriented gaussoids
by listing the six signs for $\mathcal{A}$ in the order
$\,a_{12},a_{13},a_{23}, a_{12|3},a_{13|2}, a_{23|1}$:
\small
\[
  \begin{matrix}
    0----\,- & \qquad 0-++-+ & \qquad 0+-++- & \qquad 0++-+\,+ \\
  -\,0---- & \qquad  +\,0-++- & \qquad  -\,0+-++ & \qquad   +\,0++-+   \\
   --0--\,- & \qquad     -+0-+\,+ & \qquad  +-0+-\,+ & \qquad ++0++\,-
  \end{matrix}
\]
\normalsize
The first oriented gaussoid $\,\,0----\,- \,\,$ is realized by the
symmetric $3 \times 3$-matrix $\Sigma $ with
$ (p_1,p_2,p_3,a_{12},a_{13},a_{23}) = (2,2,2,0,-1,-1)$.
Matrices for the other $11$ oriented gaussoids in this
class are obtained by relabeling and
setting $\Sigma \mapsto D \Sigma D$, where $D = {\rm diag}(\pm 1 , \pm 1 , \pm 1)$.~$\diamondsuit$
\end{example}

\begin{corollary}
For $n = 3$ there are $51$ oriented gaussoids, in $7$ symmetry classes.
These are all realizable. Among them are $20$ uniform oriented gaussoids, in
$3$ symmetry classes.  Among these $20$, there are $8$ positive gaussoids.
\end{corollary}

The following table exhibits the seven classes. The first column gives
a positive definite symmetric $3 \times 3$-matrix $\Sigma$ that
realizes the first oriented gaussoid in that symmetry class.
\small
\[
\begin{matrix}
\hbox{$(p_1,p_2,p_3,a_{12},a_{13},a_{23})$} & \hbox{Symmetry class of oriented gaussoids}  & {\rm Size} \\
(2, 2, 2, 1, 1, 1) &
+\!+\!+\!+\!+ +,  \,+\!-\!-\!+\!- -,\, -\!-\!+\!-\!- +, \, -\!+\!-\!-\!+ - & 4 \\
(3, 5, 1, 1, 1, 2) &
+\!+\!+\!-\!+ +, \, +\!-\!-\!-\!- -,\, -\!-\!+\!+\!-+, \,\ldots,
 \, -\!-\!+\!-\! -- & 12 \\
(6, 9, 6, -1, -1, -7) &
-\! -\! -\! -\! --,\, +\! +\! -\! +\! +\,-,\, -\! +\! +\!- \! ++, \,+\!-\!+\! +\! -+ & 4 \\
( 4,3  ,3  ,2 ,2 ,1 ) &  + \! +\! +\! +\! +0 ,\, +\! +\! +\! +\! 0 + ,\,
+ \!+ \!+ 0 \!+ \!+ ,\,\ldots\,,\, -\! -\! + \!-\! - 0
 & 12 \\
 (2, 2, 2, 0, -1, -1) & 0\! -\! -\! -\! --, \, 0\! -\! +\! +\!- +,\,
\ldots  \qquad \hbox{(Example \ref{ex:ogclass3})} & 12 \\
(3, 2, 2, 0, 0, 1) &
\,\, 00\!+\!00+, 00\!-\!00-, -00\!-\!00, +00\!+\!00, 0\!-\!00\!-\!0,  0\!+\!00\!+\!0 & 6 \\
(1, 1, 1, 0, 0, 0) &
000000 & 1
\end{matrix}
\]
\normalsize
See Theorem~\ref{thm:classifyOriented4} for the classification in the $n=4$ case.

We now shift gears and focus on positive gaussoids. In analogy to the
situation for positroids~\cite{ARW}, all positive gaussoids are
realizable and their realization spaces are very nice.

Let $\Gamma = ([n],E)$ be an undirected simple graph with vertex set
$[n] = \{1,2,\ldots,n\}$. This defines a gaussoid $\mathcal{G}_\Gamma$
by taking all the conditional independence statements that hold for
the graphical model~$\Gamma$.  To be precise, an unknown $a_{ij|K}$
lies in $\mathcal{G}_\Gamma$ if and only if every path from vertex $i$
to vertex $j$ in $\Gamma$ uses at least one of the vertices in~$K$.
Thus $a_{ij} \in \mathcal{G}_\Gamma$ when $i$ and $j$ are in separate
connected components of $\Gamma$, and
$a_{ij|[n]\backslash \{i,j\}} \in \mathcal{G}_\Gamma$ when $\{i,j\} \not\in E$.

\begin{theorem}\label{thm:posGaussoids}
Fix a positive integer $\,n$.  There are exactly $\,2^{\binom{n}{2}}$
positive gaussoids $\mathcal{G}_\Gamma$,  one for each of the graphs
$\Gamma = ([n],E)$.  These gaussoids are all realizable.  The space of
covariance matrices $\Sigma$ that realize $\mathcal{G}_\Gamma$ is
homeomorphic to an open ball of dimension~$|E|+n$.
\end{theorem}

\begin{proof} We first show that $\mathcal{G}_\Gamma$ supports a
positive gaussoid.  Our argument follows
\cite[Proposition~6.3]{fallat}.  Let $A = (a_{ij})$ be the adjacency
matrix of~$\Gamma$, with $a_{ij} = 1$ if $\{i,j\} \in E$ and
$a_{ij} = 0$ otherwise.  Take $\Sigma = (t \cdot {\rm Id}_n - A)^{-1}$
for sufficiently large $t > 0$.  Then $\Sigma$ is positive definite
and all its almost-principal minors are nonnegative.  Indeed,
$\Sigma^{-1}$ is an {\em M-matrix}, i.e.~it is
a positive definite matrix
whose off-diagonal entries are nonpositive.
 By \cite[Theorem~2]{karlin}, all
partial correlations of the associated Gaussian distribution are
nonnegative.  Following \cite{LUZ}, this is precisely what it means
for a distribution to be ${\rm MTP}_2$.  Hence, all
$a_{ij|K}$ are nonnegative  for our matrix $\Sigma$.  Moreover, by \cite[Theorem~6.1]{fallat},
the distribution given by $\Sigma$ is faithful to the graph $\Gamma$,
i.e.~a principal minor $a_{ij|K}$ is zero if and only if it lies
in~$\mathcal{G}_\Gamma$.

Using the same line of reasoning, we can show that the realization
space of $\mathcal{G}_\Gamma$ is homeomorphic to an open ball of
dimension~$|E|+n$. Indeed, if $\Sigma$ is any covariance matrix with
gaussoid $\mathcal{G}_\Gamma$, then $\Sigma^{-1}$ is an M-matrix with
support $\Gamma$.  The set of all such matrices is a (relatively open) convex cone of
dimension $|E|+n$.  It is the face indexed by $\Gamma$
of the cone of all M-matrices. That cone is denoted by
$\mathcal{M}^p$ in \cite[\S~2]{LUZ}. It has dimension
 $\binom{n}{2}+n$, and it is open in the ambient space of symmetric matrices.
 Note that the cone $\mathcal{M}^p$
is the realization space of the
strictly positive gaussoid $+\!+\!+\!+ \, \cdots \, + $, for the
complete graph $\Gamma = K_n$.

Matrix inversion
defines a homeomorphism from the aforementioned relatively open face
of $\mathcal{M}^p$ onto a subset of
$\RR^{\mathcal{P} \cup \mathcal{A}}$ that is topologically a ball of
dimension $|E|+n$. Its image in the positive part of the variety
$V(J_n)$ is a semialgebraic stratum of dimension $|E|+n$.

To complete the proof, let us now assume that $\mathcal{G}$ is an
arbitrary positive gaussoid. A~priori we do not know that
$\mathcal{G}$ is realizable.  We must prove that $\mathcal{G}$ equals
$\mathcal{G}_\Gamma$ for some graph $\Gamma = ([n],E)$.  By
\cite[Theorem~1]{sadeghi}, it suffices to check that $\mathcal{G}$ is
a singleton-transitive compositional graphoid.  Equivalently, by
\cite[Corollary~7]{sadeghi}, we must verify that the edge trinomials
imply the three axioms {\em singleton-transitivity}, {\em
intersection}, and {\em upward-stability}.  Singleton-transitivity is
equivalent to the gaussoid axiom~\ref{it:G4}.  Intersection
is~\ref{it:G2} and thus these two axioms hold for~$\mathcal{G}$.
Upward stability says that $a_{ij|L} \in \mathcal{G}$ implies
$a_{ij|kL} \in \mathcal{G}$.  This follows from the trinomial
$a_{ij|L} p_{kL} - a_{ik|L} a_{jk|L} - a_{ij|kL} p_L $ in $ J_n$.
Indeed, since $p_{kL}$ and $p_L$ are positive and the middle product is nonnegative,
we see that $a_{ij|L} = 0$ implies  $ a_{ij|kL}  = 0$.
\end{proof}

\begin{remark}
The oriented gaussoids that result from positive ones by reorientation
correspond to {\em signed ${\rm MTP}_2$ distributions}.  We refer to
\cite[\S~5]{LUZ} and the references given there.
\end{remark}

\section{Quadratic Relations}\label{s:quadrics}

In this section we return to the ideal $J_n$ of relations among
principal and almost-principal minors, and we derive a conjectural
characterization of its minimal generators.  We begin by discussing
the extent to which the trinomials suffice to generate.  Let
$T_n $ denote the ideal in  $\RR[\mathcal{P} \cup \mathcal{A}]$ that is generated
by all square trinomials \eqref{eq:trinomial2face} and all edge trinomials~\eqref{eq:trinomialedge}.

\begin{example} The ideal $T_3$ is generated by $18 = 6+12$ quadratic
trinomials, displayed in Example~\ref{ex:J3}.  It is
radical and its prime decomposition has five components:
\[
T_3 \,\, = \,\, J_3 \,\cap \,P_{\emptyset,123}  \,\cap \,
P_{1,23}  \,\cap \,P_{2,13} \,\cap \, P_{3,12}.
\]
Each associated prime $P_{I,J}$ is generated by $12$ of the $14$
unknowns in $\mathcal{P} \cup \mathcal{A}$.  The two unknowns not in
$P_{I,J}$ are $p_I$ and $p_J$.  The variety $V(P_{I,J})$ is a coordinate line
$\PP^1 $ in $\PP^{13}$.  \hfill $\diamondsuit$
\end{example}

We show that $T_n$ becomes equal to the prime
ideal~$J_n$ after inverting all unknowns in~$\mathcal{P}$.

\begin{proposition} \label{prop:intorus} The ideal $J_n$ is an
associated prime of the trinomial ideal~$T_n$.  Every other associated prime
of $\,T_n$ contains at least one of the $2^n$ unknowns $p_I \in \mathcal{P}$.
\end{proposition}

\begin{proof}
Let $R = \RR[\mathcal{A} \cup \mathcal{P}^{\pm 1}]$ denote the partial
Laurent polynomial ring obtained  from $\RR[\mathcal{A} \cup \mathcal{P}]$
by adjoining $p_I^{-1}$ for all $I \subseteq [n]$. Consider the ideal
$T_n R$ in $R$.  Modulo this ideal,
\[
  p_{ijK} \,=\,  p_{iK} p_{jK} p_K^{-1}\, - \,a_{ij|K}^2 p_K^{-1}
  \quad
  \hbox{and}
  \quad
  a_{ij|kL} \,=\, p_{kL} a_{ij|L} p_L^{-1} \, - \, a_{ik|L}
 a_{jk|L} p_L^{-1} .
\]
These relations express each principal or almost-principal minor
of size $\geq 2$ as a Laurent polynomial in the entries of the symmetric matrix $\Sigma$.
This shows that $R/T_n R $ is isomorphic to a partial Laurent polynomial ring in $\binom{n+1}{2}+1$ unknowns.
The same reduction argument works for the ideal $J_n$. In symbols, we have
the following isomorphism of $\RR$-algebras:
\[
  R/T_n R \,\simeq \,R/J_n R \,\simeq\, \RR\bigl[\,p^{\pm 1}, p_1^{\pm 1},p_2^{\pm 1},\ldots, p_n^{\pm 1},
  a_{12}, a_{13}, \ldots, a_{n-1,n} \bigr].
\]
We conclude that $T_n R $ equals the prime ideal $J_n R$ in $R$, and this proves the assertion.
\end{proof}

In Theorem~\ref{thm:quadricgens} we describe all quadrics in the ideal
$J_n$. We believe that these generate~$J_n$.

\begin{conjecture} \label{conjgens}
The ideal $J_n$ is generated by its quadrics, listed explicitly in
Theorem~\ref{thm:quadricgens}.
\end{conjecture}

At present we can only show that this conjecture holds {\em scheme-theoretically}, i.e.~the ideal
generated by all quadrics in $J_n$ agrees with the homogeneous prime ideal $J_n$  in all sufficiently large degrees.
The following proof of this result was suggested to us by  Mateusz Micha{\l}ek.

\begin{proposition}
The projective variety $V(J_n)$ of principal and almost-principal minors of symmetric 
$n \times n$-matrices is defined scheme-theoretically by the quadrics in its ideal $J_n$.
\end{proposition}

\begin{proof}
Let $V = \RR^{2^{n-2}(4 + \binom{n}{2})}$ and let $\PP(V)$ be the projective space whose coordinates are $\mathcal{P} \cup \mathcal{A}$. Consider the two subschemes $X, \hat{X} \subset \PP(V)$ defined respectively by $J_n$ and $\hat{J}_n := \langle(J_n)_2\rangle$, the ideal generated by the quadratic part of $J_n$. By construction, $X \subseteq \hat{X}$. Our goal is to prove that equality holds. To do this, first note that both subschemes are contained in $\bigcup_{I \subseteq [n]} D(p_I)$,
where $D(p_I) = \{ p_I \not= 0 \}$ is the affine chart given by the principal minor $p_I$.
 Indeed, since both ideals $J_n$ and $\hat{J}_n$ contain the square trinomials 
 $a_{ij|K}^2 + p_Kp_{ijK} - p_{iK}p_{jK}$, there is no 
 closed point of either subscheme whose $p$-coordinates are all equal zero.

The action of ${\rm SL}_2(\RR)^n$ induces canonical isomorphisms $D(p_I) \cap X \cong D(p_{\emptyset}) \cap X$ and $D(p_I) \cap \hat{X} \cong D(p_{\emptyset}) \cap \hat{X}$ for every $I \subseteq [n]$. 
It is hence enough to prove that the affine schemes $D(p_{\emptyset}) \cap X $
and $ D(p_{\emptyset}) \cap \hat{X}$ are equal. These affine schemes are
defined by the ideals obtained from $J_n$ and $\hat{J}_n$ by setting $p_\emptyset = 1$.
We claim that these two dehomogenized ideals are equal.

The $1$-minors of $\Sigma$
are algebraically independent modulo $J_n|_{p_\emptyset = 1}$.
It is then enough to show that all variables corresponding to minors of $\Sigma$ of size two or higher can be rewritten in terms of
the $1$-minors by dehomogenizing certain quadrics in $J_n$. One sees this as follows.

Given $i$, $j \in [n]$ and $L \subseteq [n]\backslash \{i, j\}$, consider the variable $p_{ijL}$, which is associated with a principal $(|L|+2)$-minor of $\Sigma$. Lowering the square trinomial 
$p_{\emptyset}p_{ij} - p_ip_j + a_{ij|\emptyset}^2$ by the index set $L$ yields a quadric
that contains $p_{\emptyset}p_{ijL}$ and whose other terms involve only variables corresponding to minors of $\Sigma$ of size $|L|+1$ and lower. (Here, by lowering we mean the application of the lowering operators $\ell_k$ that are defined after the statement of Theorem \ref{thm:quadricgens} below.) The analogous statement for the variable $a_{ij|kL}$ corresponding to an almost-principal minor is obtained by lowering the edge trinomial $p_{\emptyset}a_{ij|k} - p_ka_{ij|\emptyset} + a_{ik|\emptyset}a_{jk|\emptyset}$ by the index set $L \subseteq [n] \backslash \{i, j, k\}$. Thus, on $D(p_\emptyset)$, we can use quadrics in $J_n$
to rewrite every larger minor as a polynomial in the $1$-minors. Thus
$J_n|_{p_\emptyset = 1}$ is generated by dehomogenized quadrics.
\end{proof}

\begin{theorem} \label{thm:quadricgens} The space of all quadrics in the ideal $J_n$ is a $G$-module of dimension
\begin{equation}
\label{eq:magicdim}
{\rm dim}\bigl( (J_n)_2 \bigr) \quad = \quad
3^{n-2}\binom{n}{2} \,+\,\, 4\sum_{m=3}^{n}3^{n-m}\binom{n}{m}\binom{m}{2} 
\, +\, \,\sum_{k=2}^{\lfloor{\frac{n}{2}}\rfloor}2k \cdot 3^{n-2k}\binom{n}{2k}.
\end{equation}
The following four classes of quadrics and their images under $S_n$
are the highest weight vectors for the distinct irreducible
representations occurring in the $G$-module $(J_n)_2$:
\begin{gather}
p_{12}p \, -\, p_1p_2 \,+\, a_{12}^2 \tag{$i$}\label{q1}\\
\sum_{L \subseteq [m] \backslash \{1, 2\}} \!\!\!\! (-1)^{|L|}p_La_{12|L^c} + \sum_{j=3}^m
\sum_{K \subseteq [m] \backslash \{1, 2, j\}} \!\!\! \!\! (-1)^{|K|}a_{1j|K}a_{2j|K^c} \quad
\hbox{\rm for $\,3 \leq m \leq n$, $\,m$ odd}
\tag{$ii$}\label{q2}\\
\end{gather}
\begin{gather}
\sum_{j=3}^m\sum_{K \subseteq [m] \backslash \{1, 2, j\}} \!\!\!\! (-1)^{|K|}a_{1j|K}a_{2j|K^c} \quad
\hbox{\rm for $\,4 \leq m \leq n$, $\,m$ even} \tag{$iii$}\label{q3}\\
\sum_{\substack{(L,L') \,{\rm partition} \\ {\rm  of }\,[m]}} \!\!\!\!\! (-1)^{|L|}p_Lp_{L'} \,\,+\,\,
2 \cdot \sum_{j=2}^m \sum_{\substack{(K, K')\, {\rm partition} \\ {\rm of }\,[m] \backslash \{1,j\}}} \!\!\! \!\!
(-1)^{|K|}a_{1j|K}a_{1j|K'} \,\,\, \, \hbox{\rm for $3 < m \leq n$, $m$ even}.
\tag{$iv$}\label{q4}
\end{gather}
\end{theorem}

To find these quadrics, we used the Lie algebra $\mathfrak{g} = \mathfrak{sl}(2, \mathbb{R})^{\oplus n}$
of the group $ G = {\rm SL}_2(\RR)^n$.    As $n$ increases, so does the number of quadrics.
However, just a small fraction is new: most come from earlier ones via
lowering operators in~$\mathfrak{g}$.  These are described in \cite[Remark~III.16]{lukediss}.

The $k$-th \emph{lowering operator} ${\ell}_k$ is the following endomorphism
of $W = W_{\rm pr} \oplus \bigoplus_{i,j} W_{\rm ap}^{ij}$:
\[\begin{array}{lcr}
p_L \mapsto \begin{cases}p_{L \cup \{k\}} & \text{if } k \notin L\\ 0 & \text{otherwise} \end{cases} & \qquad &
a_{ij|L} \mapsto \begin{cases}a_{ij|L \cup \{k\}} & \text{if } k \notin L \cup \{i, j\} \\ 0 & \text{otherwise} \end{cases}
\end{array}\]
Similarly, the $k$-th \emph{raising operator} $r_k$  acts on $W$ as follows:
\[\begin{array}{lcr}
p_L \mapsto \begin{cases}p_{L \setminus \{k\}} & \text{if } k \in L\\ 0 & \text{otherwise} \end{cases} & \qquad &
a_{ij|L} \mapsto \begin{cases}a_{ij|L \setminus \{k\}} & \text{if } k \in L  \\ 0 & \text{otherwise} \end{cases}
  \end{array}
\]
These operators are extended to ${\rm Sym}_2(W)$ by the Leibniz rule~\cite[\S 8.1]{FH}:
\[
\ell_k(vw) \,\,=\,\, \ell_k(v) \cdot w \,+\, v \cdot \ell_k(w) \quad\text{ and }\quad
r_k(vw) \,\,=\,\, r_k(v) \cdot w \,+\, v \cdot r_k(w).
\]
These endomorphisms of ${\rm Sym}_2(W)$ restrict to the $G$-submodule $(J_n)_2$.

\bigskip
\begin{table}[ht]
\centering
\begin{tabular}{c||ccccccc}
 $n$ & 2 & 3 & 4 & 5 & 6 & 7 & 8 \\
 \hline
 \# quadrics in $J_n$ &  1 &  21  & 226  & 1810 & 12261 & 74613 & 421716
\end{tabular}
\caption{The number of quadratic generators of~$J_n$.}
\label{tab:expgens}
\end{table}

\begin{remark} \label{rem:highest} A nonzero polynomial lies in the
kernel of all possible raising (respectively, lowering) operators if
and only if it is a highest (respectively, lowest) weight vector.
\end{remark}

\begin{figure}[h]
\centering
\adjustbox{scale=.8,center}{%
\begin{tikzcd}[column sep=small]
& & (0,0,2,2) \arrow[swap]{ld}{{\ell}_3} \arrow{rd}{{\ell}_4} & & \\
& (0,0,0,2) \arrow{ld}[swap]{{\ell}_3} \arrow{rd}{{\ell}_4} & & (0,0,2,0) \arrow{ld}[swap]{{\ell}_3} \arrow{rd}{{\ell}_4} & \\
(0,0,-2,2) \arrow{rd}{{\ell}_4} & & (0,0,0,0) \arrow{ld}[swap]{{\ell}_3}\arrow{rd}{{\ell}_4} & & (0,0,2,-2) \arrow{ld}[swap]{{\ell}_3}\\
& (0,0,-2,0) \arrow{rd}{{\ell}_4}& & (0,0,0,-2) \arrow{ld}[swap]{{\ell}_3} & \\
& & (0,0,-2,-2) & &
\end{tikzcd}
}
\caption{A visualization of the $G$-module $S_{0022}$ inside $(J_4)_2$. Arrows pointing down and left represent lowerings via ${\ell}_3$,
 while arrows pointing down and right represent lowerings via
 ${\ell}_4$. \label{fig:loweringposet}}
\end{figure}
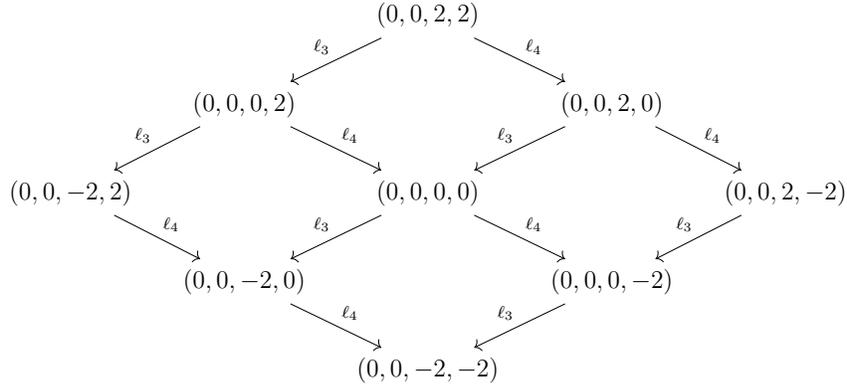

 \begin{example} \label{ex:monomial}
Consider the square trinomial $\,p_{12}p \, -\, p_1p_2 \,+\, a_{12}^2 \,$
in~\eqref{q1}.  If $n=2$, then it has weight $00$.  For $n>2$ the
weight is $0022\ldots 2$.  The quadric~\eqref{q1} is a
highest weight vector since it is annihilated by the raising operators  $r_k$.  It generates the
$3^{n-2}$-dimensional $G$-module $S_{002\ldots 2} $ inside $ (J_n)_2$.
To get an $\RR$-basis of this $G$-module, we apply all lowering
operators, which yields $3^{n-2}$ quadrics.  For instance, if $n=4$
then lowering via $\ell_3$ yields
$p_3p_{12} + p p_{123} - p_{13}p_2 - p_1p_{23} +
2a_{12}a_{12|3}$.  Taking into account all $\binom{n}{2}$ permutations
of  $002\ldots 2$, we find $3^{n-2}\binom{n}{2}$ quadrics
originating from~\eqref{q1}.  This explains the first summand
in~\eqref{eq:magicdim}.  \hfill $\diamondsuit$
\end{example}

\begin{proof}[Proof of Theorem \ref{thm:quadricgens}]
The proposed quadrics lie in the kernel of the raising
operators and hence are highest weight vectors by
Remark~\ref{rem:highest}.  The count in~\eqref{eq:magicdim} is
explained by working through the action of the lowering operators on
\eqref{q1}-\eqref{q4}.
This was illustrated in Example \ref{ex:monomial} above.
Specifically, each of the $\binom{n}{2}$ quadrics
in~\eqref{q1} contributes $3^{n-2}$ quadrics to~$J_n$.  For fixed $m$,
each of the $\binom{n}{m}\binom{m}{2}$ quadrics of types~\eqref{q2}
and~\eqref{q3} contributes $4\cdot 3^{n-m}$ quadrics.  Similarly, for fixed
$m$, each of the $m \binom{n}{m}$ highest weight quadrics of type~\eqref{q4}
gives rise to $3^{n-m}$ linearly independent quadrics in~$J_n$.

We next prove that quadrics of type~\eqref{q4} lie in~$J_{n}$,
i.e.~they map to zero in
$\RR[\Sigma] = \RR[\sigma_{11},\sigma_{12},\ldots,\sigma_{nn}$].  The proofs
for  \eqref{q2} and  \eqref{q3}  are similar, but simpler.
Without loss of generality we assume that $m=n$ (and hence $n$ is
even).  Any monomial in the quadrics~\eqref{q4} maps to some monomial
of~$\det \Sigma$.  Specifically, $p_Lp_{L'}$ and each
$a_{1j|K}a_{1j|K'}$ map to monomials
$\label{eq:monomial} \sigma_{\pi} :=
\sigma_{1,\pi(1)}\sigma_{2,\pi(2)} \cdots \sigma_{n,\pi(n)}$ where
$\pi \in S_{n}$.  To show this for the $a$ monomials,
 it is helpful to arrange the rows of $\Sigma$ as
$(1,K,j,K')$ and the columns as $(j,K,1,K')$. With this arrangement,
\begin{equation}\label{eq:sigmasplit}
\Sigma \,\,= \,\,
\begin{pmatrix}
\Sigma_{1K\times jK} & \Sigma_{1K\times 1K'}\\
\Sigma_{jK'\times jK} & \Sigma_{jK' \times 1K'}
\end{pmatrix}
\end{equation}
Then $a_{1j|K}a_{1j|K'}$ maps to a monomial in the expansion of
$\,\det \Sigma_{1K\times jK} \cdot \det\Sigma_{jK' \times 1K'}$.

Fix a monomial $\sigma_\pi$ and let $C_1C_2 \cdots C_s$ be the cycle
decomposition of~$\pi$.  Assume $1\in C_1$.  Let $r$ be the number
of cycles of length $ \geq 3$.  Since $\Sigma$ is symmetric,
replacing a cycle $C_i$ in $\pi$ by its inverse does not
change~$\sigma_{\pi}$.
The monomial $\sigma_{\pi}$ appears in the image of $p_Lp_{L'}$ whenever
$(C_1, \ldots, C_s)$ refines the partition~$(L, L')$, and
 it appears with the same sign as in $\det \Sigma$.  The monomial
$\sigma_{\pi}$ appears in the image of~$a_{1j|K}a_{1j|K'}$ only if
$C_{1}$ contains $j$, as seen from~\eqref{eq:sigmasplit}.
Additionally, each of $C_2,C_3,\ldots,C_s$ must be contained in either
$K$ or~$K'$.  Finally, if $C_{1}= (1,i_{2},\dots,i_{l},j,i_{l+2},\dots, i_{t})$, then
$\{i_{2},\dots,i_{l}\} \subseteq K$ and
$\{i_{l+2},\dots,i_{t}\} \subseteq K'$.  These three properties together
characterize the monomials $\sigma_\pi$ that appear in the image of~\eqref{q4}.
 According to our sign convention for the $a_{ij|K}$, the product $a_{1j|K}a_{1j|K'}$
has a global minus with respect to $\det\Sigma$, again visible
from~\eqref{eq:sigmasplit} as the columns $1$ and $j$ have been
exchanged.

Assume first that $\pi$ has an even cycle $\widetilde{C}$.  Since
$n$ is even, there is another even cycle in $\pi$, and we can assume
$1\notin \widetilde{C}$.  Consider the following matching of terms
of~\eqref{q4}:
\begin{gather*}
(p_Lp_{L'}, p_{L \cup \widetilde{C}}p_{L'\setminus \widetilde{C}})
\text { if } \widetilde{C} \nsubseteq L, \quad
(p_Lp_{L'}, p_{L \setminus \widetilde{C}}p_{L' \cup \widetilde{C}}) \text{ otherwise, } \\
(a_{1j|K}a_{1j|K'}, a_{1j|K \cup \widetilde{C}}a_{1j|K' \setminus
\widetilde{C}}) \text{ if } \widetilde{C} \nsubseteq K, \quad
(a_{1j|K}a_{1j|K'}, a_{1j|K \setminus \widetilde{C}}a_{1j|K' \cup
\widetilde{C}}) \text{ otherwise}.
\end{gather*}
Since $\widetilde{C}$ has odd cardinality, the matched terms differ in
signs and cancel in the image.

Let now $\pi$ be a product of odd cycles and denote
$C_1 = \{1, i_2, i_3, \ldots, i_{2w}\}$.  We can again produce
cancellations by matching the following terms for any $1< u \leq w$:
\[
\bigl( \,a_{1i_{2u-1}|K}a_{1i_{2u-1}|K'}\,,\,\, a_{1i_{2u}|K \cup \{i_{2u-1}\}}a_{1i_{2u}|K' \setminus \{i_{2u-1}\}} \,\bigr).
\]
After subtraction of the matched terms, the remaining terms are of
the form $a_{1i_{2}|K}a_{1i_{2}|K'}$.  We now count the occurrences of
$\sigma_{\pi}$ in the images of these.  For the $p$-part, there are
$2^{s-1}$ partitions $(L,L')$ of $[n]$ that coarsen the cycles
of~$\pi$.  For each of these, there are $2^r$ copies
of~$\sigma_{\pi}$ because there are $r$ cycles with cardinality~$\ge 3$.
In total, the coefficient of $\sigma_{\pi}$ in the image of
$\sum_{(L,L')} p_{L}p_{L'}$ is~$2^{s-1+r}$.
For the $a$-part, we distinguish two cases.  First, if $C_1$ is a
transposition, there are $2^{s-2}$ partitions $(K,K')$ that coarsen
$(C_{2},\dots,C_{s})$.  Again, $\sigma_\pi$ appears $2^r$ times from
reorientations of cycles of length~$\ge 3$.  Thus the total count is
$2^{(s-2)+r} = 2^{s+r-2}$.  Now, if $C_{1}$ is not a transposition, then
there are $r-1$ cycles in $C_2, \ldots, C_s$ with cardinality~$\ge 3$.
Then $2^{r-1}$ copies of $\sigma_\pi$ appear for each of the $2^{s-1}$
coarsenings.  Again, the total count is $2^{(r-1)+(s-1)} = 2^{s+r-2}$.
In \eqref{q4}, the coefficient $2$ corrects the count, and the sign of
the monomials in the $a$ terms has a global minus relative to
the determinant of~$\Sigma$.


We now show that the $G$-module $(J_n)_2$ is spanned by
 the quadrics  \eqref{q1}-\eqref{q4}.
 Let $f$ be any quadric in~$J_n$.
We can assume that $f \notin J_m$ for $m < n$ and $f$
is a highest weight vector. A~priori the weight of
$f$ is an element of $\{-2, -1, 0, 1, 2\}^n$. We claim that it is
in~$\{0,1\}^n$.  To see this, assume first that $-2$ or $-1$
appears in the weight.  In this case, raising $f$ at the
corresponding index yields a nonzero quadric of higher weight.  Moreover, an entry
$2$ can only appear if the corresponding index appears in no variable
of $f$ and thus $f\in J_{n-1}$.

Given the weights of $a_{ij|K}$ and $p_L$, the only possible weights
for $f$ (up to permutation) are $111100 \ldots 0$, $1100 \ldots 0$,
and~$00 \ldots 0$.  The following are general quadrics for these weights:
\begin{align*}
111100 \ldots 0\colon &\quad \sum_{K \subseteq [n] \setminus \{1, 2, 3, 4\}} d_K \cdot a_{12|K}a_{34|K^c}\\
1100 \ldots 0\colon &\quad \sum_{L \subseteq [n] \setminus \{1,2\}} c_L \cdot p_L a_{12|L^c}
\,\,+\,\, \sum_{j=3}^n \sum_{K \subseteq [n] \setminus \{1, 2, j\}} d^{(j)}_K \cdot a_{1j|K}a_{2j|K^c}\\
00 \ldots 0\colon &\quad \sum_{\substack{(L, L')\,{\rm partition} \\ {\rm  of }\,[n]}}c_L \cdot p_L p_{L'}
\,\,+\, \sum_{\substack{i, j \in [n] \\ i < j}}\sum_{\substack{(K, K')\, {\rm partition} \\
{\rm of }\,[n] \setminus \{i,j\}}} \!\!\!\!\! d^{(ij)}_{K} \cdot a_{ij|K}a_{ij|K'}.
\end{align*}
That $f$ lies in the kernel of all raising operators
imposes conditions on the coefficients $c,d$.  In particular, all
coefficients in an inner sum (like $c_L$ or $d^{(j)}_K$ for a fixed $j$)
differ by at most a sign, in an alternating fashion.  More precisely,
for each $i \in L$ one has $c_{L\backslash \{i\}} + c_L = 0$ and
hence, inductively, $c_L = (-1)^{|L|}c_{\emptyset}$.
This implies that $c_{\emptyset} = c_{[n]} = (-1)^n c_{\emptyset}$,
 and a similar statement holds for each $d^{(ij)}_{\emptyset}$.
 We conclude that,  when $n$ is odd, there can be no quadric
of weight $00\ldots 0$ that lies in $J_n$ and satisfies our hypotheses.

By the first part of the proof, the quadrics of types
\eqref{q1}-\eqref{q4} do arise.  It is therefore enough to prove that
there are no further linearly independent quadrics in each weight.  To
do this, we look at the image of $f$ in~$\mathbb{R}[\Sigma]$.

\begin{itemize}
\item $111100 \ldots 0$: The monomial
$\,\sigma_{1,2} \sigma_{3,5} (\prod_{i=5}^{n-1}\sigma_{i, i+1})
\sigma_{n,4}\,$ is among the terms in the image of $a_{12}a_{34|5\ldots n}$ only and thus
$d_\emptyset = 0$.  Hence no quadrics arise.
\item $1100 \ldots 0$, $n$ odd: Similarly,
$\sigma_{1,3}\sigma_{2,3} (\prod_{i=4}^{n-1}\sigma_{i, i+1})
\sigma_{n,4}$ yields $c_\emptyset=d_\emptyset^{(3)}$.  Permuting
suitably we find $c_\emptyset = d_\emptyset^{(j)}$ for each $j$, and
hence there is at most one quadric of this weight.
\item $1100 \ldots 0$, $n$ even:
$\sigma_{1,2}(\prod_{i=3}^{n-1}\sigma_{i, i+1}) \sigma_{n,3}$ and
$\sigma_{1,3}\sigma_{3,4}\sigma_{4,2}(\prod_{i=5}^{n-1}\sigma_{i,
i+1}) \sigma_{n,5}$ give that $c_\emptyset=0$ and
$d_\emptyset^{(3)}= d_\emptyset^{(4)}$; permuting the indices suitably
we get that $d_\emptyset^{(i)} = d_\emptyset^{(j)}$ for all $i, j$ and
hence there is at most one quadric with this weight.
\item $00 \ldots 0$, $n \geq 4$ even: When $n \geq 4$, the preimage of
the monomial $\prod_{i=1}^{n/2}\sigma_{2i-1,2i}^2$ gives that
$2c_\emptyset = d_\emptyset^{(12)} + d_\emptyset^{(34)} + \cdots +
d_\emptyset^{(n-1, n)}$. From this relation and its permutations one
derives that
$d_\emptyset^{(ij)} + d_\emptyset^{(kl)} = d_\emptyset^{(ik)} +
d_\emptyset^{(jl)}$ for any four distinct indices $i, j, k, l$.
Consequently, all coefficients in $f$ can be expressed in terms of
$d_\emptyset^{(1j)}$ (where $j$ ranges from $2$ to $n$)
and~$d_\emptyset^{(23)}$.  Thus, the dimension of the associated
vector subspace of $(J_n)_2$ is at most~$n$.\qedhere
\end{itemize}
\end{proof}

\section{Tropical Geometry}\label{s:valuatedG}

In recent years, the theory of matroids has been linked tightly to the
emerging field of tropical geometry \cite{BB, MS}. Every matroid defines a
tropical linear space, and conversely, every tropical linear space
corresponds to a {\em valuated matroid}.  First introduced by Dress
and Wenzel \cite{DW, DW91} as a generalization of matroids, valuated
matroids are now best understood as vectors of tropical Pl\"ucker
coordinates.  For a textbook introduction to this topic see
\cite[Chapter~4]{MS}.

Tropical geometry is a combinatorial shadow of algebraic geometry over
a field with valuation. The field of real Puiseux series,
$\RR \{ \! \{ \epsilon \} \!\}$, is our primary example. This field is
ordered and it contains the rational functions $\RR(\epsilon)$.  The
unknown $\epsilon$ is positive but smaller than any positive real
number.  Covariance matrices with entries that contain $\epsilon$ can
be found in the realizations of gaussoids by Ln\v{e}ni\v{c}ka and
Mat\'u\v{s}~\cite[Table~1]{LM}. Indeed, statisticians frequently
consider Gaussian distributions that depend on a perturbation
parameter~$\epsilon$.  The development in this section represents a
systematic approach to the analysis of such distributions.

A {\em valuated gaussoid} on $[n]$ is a map
$\nu : \mathcal{P} \cup \mathcal{A} \rightarrow \RR $ such that the
minimum of $\nu(m_1),\nu(m_2),\nu(m_3)$ is attained at least twice for
every quadratic trinomial $m_1+m_2+m_3$ in $T_n$.  Here $\nu(m_i)$ is
the sum of the values of $\nu$ on the two terms in $m_i$.  In other
words, a valuated gaussoid is a point $\nu$ in the tropical prevariety
defined by the trinomials \eqref{eq:trinomial2face}
and~\eqref{eq:trinomialedge}.
Recall that $V(J_n) = V(T_n)$ in the torus by Proposition \ref{prop:intorus}.
Every point $\nu$ in the tropical variety
${\rm trop}(V(J_n)) = {\rm trop}(V(T_n))$ is a {\em realizable} valuated gaussoid.
The distinction between valuated gaussoids and those that are realizable mirrors
the distinction in \cite[\S~4.4]{MS} between tropical linear spaces
and tropicalized linear spaces.  The former are parametrized by the
Dressian whereas the latter are parametrized by the tropical
Grassmannian.  This is the distinction between the tropical prevariety
and tropical variety defined by our trinomials.

In Section \ref{s:realizable} we encounter non-realizable valuated gaussoids for
$n \geq 4$.  Here we focus on the case $n=3$. The variety $V(J_3)$
equals the Lagrangian Grassmannian ${\rm LGr}(3,6) \subset \PP^{13}$.
That $6$-dimensional variety has a $3$-dimensional torus
action. Modulo lineality, the tropical variety
${\rm trop}({\rm LGr}(3,6))$ is a $3$-dimensional fan, hence a
$2$-dimensional polyhedral complex.

 Recall that ${\rm LGr}(3,6)$ is a linear section
of the classical Grassmannian ${\rm Gr}(3,6) \subset \PP^{19}$. That
$9$-dimensional variety has a $5$-dimensional torus action.
Modulo its lineality space, the {\em tropical Grassmannian} ${\rm trop}({\rm Gr}(3,6))$ is
a $4$-dimensional fan, hence a $3$-dimensional polyhedral complex.
It is glued from  $990$ tetrahedra and $15$ bipyramids \cite[Example~4.3.15]{MS}.
This complex is well-known to tropical geometers. A detailed description is found in
 \cite[\S~5.4]{MS}.

 The following is our main result in this section.
 In the course of proving it, we also describe the
 inclusion of ${\rm trop}({\rm LGr}(3,6))$ inside  ${\rm trop}({\rm Gr}(3,6))$,
 and we compute {\em Khovanskii bases} and  {\em Newton--Okounkov bodies} as in~\cite{KM}.
 In Corollary \ref{cor:1016} we connect to statistics by
 explaining the ${\rm MTP}_2$ distributions encoded in the
 {\em positive tropical variety} ${\rm trop}_+({\rm LGr}(3,6))$.

\begin{theorem}
\label{thm:J3tropical}
For $n=3$, all valuated gaussoid are realizable, so they are precisely
the points in the tropical Lagrangian Grassmannian
${\rm trop}({\rm LGr}(3,6))$.  The underlying $2$-dimensional
polyhedral complex has $35$ vertices, $151$ edges, and $153$ facets.
The facets come in nine symmetry classes: there are
$12+8+48+24+6+24+24+1$ triangles and $6$ quadrilaterals. Seven of the
nine facet classes represent {\em prime cones} in the sense of
Kaveh--Manon {\rm \cite[\S~5]{KM}}.
\end{theorem}

\begin{proof}[Proof and Explanation]
These results are obtained by computation. The tropical variety of $J_3$ is a pure $7$-dimensional
fan in $\RR^{14}$ whose lineality space $L$ has dimension $4$. One  dimension comes from
the usual grading, since $J_3$ is a homogeneous ideal. The others come
from the maximal torus of $G = {\rm SL}_2(\RR)^3$. Hence ${\rm trop}(V(J_3)) = {\rm trop}(
{\rm LGr}(3,6))$ is a pure $3$-dimensional fan in
$\RR^{14}/L$. The coordinates on $\RR^{14}$ are dual to a
distinguished spanning set of $\,\RR^{14} /L$:
\begin{equation}
\label{eq:fourteen}
\bigl( a_{12},a_{12|3},a_{13},a_{13|2}, a_{23}, a_{23|1}, \,p, p_1, p_{12}, p_{123}, p_{13}, p_2, p_{23}, p_3 \bigr).
\end{equation}
With this ordering of the $14$ generators, the lineality space of ${\rm trop}(J_3)$ equals
\begin{equation}
\label{eq:lineality} L \,\, = \,\, {\rm rowspace} \,
  \begin{pmatrix}
   1 &    1 &  1 &  1   &  1  &   1   & 1 &  1 &  1  &   1  &  1  &  1 &  1  &  1 \\
   0 &    0 &  0 &  0   & -1  &   1   &-1 &  1 &  1  &   1  &  1  & -1 & -1  & -1 \\
   0 &    0 & -1 &  1   &  0  &   0   &-1 & -1 &  1  &   1  & -1  &  1 &  1  & -1 \\
  -1 &    1 &  0 &  0   &  0  &   0   &-1 & -1 & -1  &   1  &  1  & -1 &  1  &  1 
  \end{pmatrix}.
\end{equation}
We use the symbols in (\ref{eq:fourteen}) to denote the corresponding spanning vectors of $\RR^{14} /L \simeq \RR^{10}$.

The $35 = 6+8+3+12+6$ rays of the fan ${\rm trop}(V(J_3))$ come in five symmetry classes: \vspace{-0.07in}
\begin{small}
\begin{itemize}
\item $6$ of type {\bf a}: $\,\{a_{12},a_{13},a_{23},a_{12|3},a_{13|2},a_{23|1}\}$ \vspace{-0.09in}
\item $8$ of type {\bf p}: $\,\{p,p_1,p_2,p_3,p_{12},p_{13},p_{23},p_{123} \}$ \vspace{-0.08in}
\item $3$ of type {\bf A}: $\,\{a_{12}+a_{12|3}+a_{13}+a_{13|2},
                                a_{12}+a_{12|3}+a_{23}+a_{23|1},a_{13}+a_{13|2}+a_{23}+a_{23|1} \}$ \vspace{-0.08in}
\item $12$ of type {\bf B}: $\, \{\,a_{12}{+}a_{12|3}{+}a_{13}{+}a_{13|2}{+}2p{+}2p_1, 
a_{12}{+}a_{12|3}{+}a_{23}{+}a_{23|1}{+}2p{+}2p_2, 
a_{13}{+}a_{13|2} $ \\ $ {+}a_{23} {+}  a_{23|1}{+}2p{+}2p_3, 
a_{12}{+}a_{12|3}{+}a_{23}{+}a_{23|1}{+}2p_1{+}2p_{12},
 a_{13}{+}a_{13|2}{+}a_{23}{+}a_{23|1}{+}2p_1{+}2p_{13}, $ \\
  $ a_{12} {+}a_{12|3}{+}a_{13}{+}a_{13|2}{+}2p_2{+}2p_{12},
 a_{13}{+}a_{13|2}{+}a_{23}{+}a_{23|1}{+}2p_2{+}2p_{23}, 
 a_{12}{+}a_{12|3}{+}a_{13}{+}a_{13|2} $ \\ $ {+}2p_3{+}2p_{13}, 
 a_{12}{+}a_{12|3}{+}a_{23}{+}a_{23|1}{+}2p_3{+}2p_{23},
  a_{13}{+}a_{13|2}{+}a_{23}{+}a_{23|1}{+}2p_{12}{+}2p_{123},  $ \\ $
  a_{12}{+}a_{12|3}{+}a_{23}{+}a_{23|1} {+}2p_{13}{+}2p_{123},
   a_{12}{+}a_{12|3}{+}a_{13}{+}a_{13|2}{+}2p_{23}{+}2p_{123}\}$ \vspace{-0.08in}
   \item $6$ of type {\bf C}: $\,\{a_{12} {+}a_{13|2}{+}2p_2{+}2p_{12}, \,
a_{23}{+}a_{12|3}{+}2p_3{+}2p_{23}, \,
a_{23}{+}a_{13|2}{+}2p_2{+}2p_{23}, $ \\ $
a_{13|2}{+}a_{23|1}{+}2p_{12}{+}2p_{123}\,,\,\,
a_{12|3}{+}a_{13|2}{+}2p_{23}{+}2p_{123}\,,\,\,
a_{12|3}{+}a_{23|1}{+}2p_{13}{+}2p_{123}\} $
          \end{itemize}
\end{small}
Each of the sums in the lists above is a vector in $\RR^{14}/L$.
For instance, the last sum in type ${\bf C}$ represents the vector
$\,( 0 ,1 ,0, 0, 0, 1, \, 0,   0 ,    0,           2,2, 0,0,0)+ L\,$
if we use  the ordering in (\ref{eq:fourteen}).

The tropical Lagrangian Grassmannian ${\rm trop}(V(J_3)) $ is the intersection of
the tropical  Grassmannian ${\rm trop}({\rm Gr}(3,6))$ with a linear space.
This intersection is computed in the $20$ Pl\"ucker coordinates
with the {\tt Macaulay2} code in Example \ref{ex:J3}.
We shall use the identification of the $20$ Pl\"ucker coordinates  with the $14$ principal
and almost-principal minors given in  (\ref{eq:identify}).

The tropical variety ${\rm trop}(V(J_3))$ has a unique
coarsest fan structure with $153$ facets. These come in $9$
orbits under the symmetries of the $3$-cube. In what follows
we list these orbits. Each facet in eight of the orbits
lies in a unique facet of
${\rm trop}({\rm Gr}(3,6))$. We name that facet in the notation of
\cite[\S~5.4]{MS}.
Facets of type {\bf ppp} lie in triangles of ${\rm trop}({\rm Gr}(3,6))$.
 Here is now the list of all
$153=12+8+48+24+6+24+24+1+6$ facets of ${\rm trop}(V(J_3))$:

\begin{small}
\begin{itemize}
\item $12$ triangles of type {\bf app}, like   $\{\,a_{12}, p_{3}, p_{123} \,\}$. They lie in tetrahedra EEEE. \vspace{-0.08in}
\item  $8 $ triangles of type {\bf ppp}, like  $\{\,p_1, p_2, p_3 \,\}$. They lie in triangles EEE.
\vspace{-0.08in}
\item $48$ triangles  of type {\bf apB}, like $\{\,a_{12|3},\, p, \,a_{12}{+}a_{12|3}{+}a_{13}{+}a_{13|2}{+}2p{+}2p_1\,\}$.
\\ They lie in tetrahedra EEEG  of the tropical Grassmannian ${\rm trop}({\rm Gr}(3,6))$.
 \vspace{-0.08in}
\item $24$ triangles of type {\bf ppC}, like $  \{\,p,\, p_{12}, \,a_{13|2}{+}a_{23|1}{+}2p_{12}{+}2p_{123}\,\} $. \\
Twelve lie in tetrahedra EEFFa, and others lie in tetrahedra EEFFb.
 \vspace{-0.08in}
\item $ 6$ triangles of type {\bf aAA}, like  $\{\,a_{12}, \,a_{12}{+}a_{12|3}{+}a_{13}{+}a_{13|2},\, a_{12}{+}a_{12t3}{+}a_{23}{+}a_{23|1}\,\}$.
\\
They lie in tetrahedra EEFFb.
\vspace{-0.08in}
\item $24$ triangles of type {\bf aAB},
like $\,\{a_{12},\, a_{12}{+}a_{12|3}{+}a_{13}{+}a_{13|2},$ \\ $
a_{12}{+}a_{12|3}{+}a_{13}{+}a_{13|2}$  $  {+}2(p_{23}{+}p_{123}) \} $.
They lie in tetrahedra EEFG.
 \vspace{-0.08in}
\item $24$ triangles of type {\bf pBC}: $\{\,p\,,\, a_{12}{+}a_{12|3}{+}a_{13}{+}a_{13|2}{+}2(p+p_1),$ \\ $
a_{12|3}+a_{13|2}+2(p_{23}+p_{123})\} $.
They lie in tetrahedra EEFG.
 \vspace{-0.08in}
\item $ 1$ triangle of type {\bf AAA}:
   $\{a_{12}{+}a_{12|3}{+}a_{13}{+}a_{13|2}, \, a_{12}{+}a_{12|3}{+}a_{23}{+}a_{23|1},$ \\
   $ a_{13}{+}a_{13|2}{+}a_{23}{+}a_{23|1}\}$.
This triangle lies in a bipyramid FFFGG.
   \vspace{-0.08in}
\item $6$ squares of type {\bf ABCB}, like $\{
a_{13}{+}a_{13|2}{+}a_{23}{+}a_{23|1}+2(p_1{+}p_{13}),\,a_{13}{+}a_{13|2}{+}a_{23}{+}a_{23|1},\, $ \\ $
a_{13}{+}a_{13|2}{+}a_{23}{+}a_{23|1}{+}2(p_2{+}p_{23}),
a_{13|2}{+}a_{23}{+}2(p_2{+}p_{23})\}$, lying in bipyramids FFFGG.
\end{itemize}
\end{small}
We conclude that all $7$ combinatorial types of valuated matroids in ${\rm trop}({\rm Gr}(3,6))$
are realized by valuated gaussoids. This is similar to the result
of Brodsky, Ceballos, and Labb\'e in~\cite{BCL}.

Each of our $153$ facets supports a monomial-free initial ideal
${\rm in} _\nu(J_3)$.  Here $\nu \in \RR^{14}$ is a vector in the
relative interior of that $7$-dimensional cone, and the initial ideal
is understood in the sense of \cite[\S~2.4]{MS}.  For the facets of
type {\bf ppp} and {\bf ABCB}, the initial ideal ${\rm in} _\nu(J_3)$ is not a
prime ideal.  For the other seven types, the initial ideal ${\rm in} _\nu(J_3)$
is toric and hence prime.  In those cases the $14$ coordinates form a
{\em Khovanskii basis} of our algebra, by the results of~\cite{KM}.
The list of types above thus classifies the {\em toric degenerations}
of the Lagrangian Grassmannian $V(J_3)$ in $\PP^{13}$, and from
${\rm in} _\nu(J_3)$ we can identify the corresponding {\em
Newton-Okounkov bodies}.  We illustrate this for type {\bf app} in the
example that follows.
\end{proof}

\begin{figure}[ht]
  \begin{center}
    \begin{tikzpicture}[scale=1.6, rotate=-5]
      \tikzset{vertex/.style = {shape=circle,fill=black,draw,minimum
          size=6pt, inner sep=0pt}}
      \tikzset{edge/.style = {line width=2pt}}

      \node [vertex, label=below:$p$] (p) at (0,-2) {};
      \node [vertex, label=left:$p_{1}$] (p1) at (-2,0) {};
      \node [vertex, label=right:$p_{2}$] (p2) at (2,0) {};
      \node [vertex, label=above:$p_{12}$] (p12) at (0,2) {};
      \node [vertex, label=below left:$p_{13}$] (p13) at (-.75,0) {};
      \node [vertex, label=below right:$p_{23}$] (p23) at (0.75,0) {};

      \draw (p) -- (p13) node[vertex, midway, label=left:$a_{13}$] (a13) {};
      \draw (p) -- (p23) node[vertex, midway, label=right:$a_{23}$] (a23) {};
      \draw (p12) -- (p13) node[vertex, midway, label={184:$a_{23|1}$}] (a231) {};
      \draw (p12) -- (p23) node[vertex, midway, label=below right:$a_{13|2}$] (a132) {};
      \draw (p13) -- (p23) node[vertex, midway, label=below:$a_{12|3}$] (a123) {};

      \draw (p) -- (p1) -- (p12) -- (p2) -- (p);
      \draw (p1) -- (p13);
      \draw (p23) -- (p2);

    \end{tikzpicture}
  \end{center}
  \vspace{-0.2in}
  \caption{\label{fig:zwei} Schlegel diagram of a $3$-polytope.     Its join with the triangle
    $\{a_{12}, p_{3}, p_{123} \}$ is a $6$-polytope. This is the Newton-Okounkov body
for a toric degeneration of ${\rm LGr}(3,6)$ in $ \PP^{13}$.    }
\end{figure}
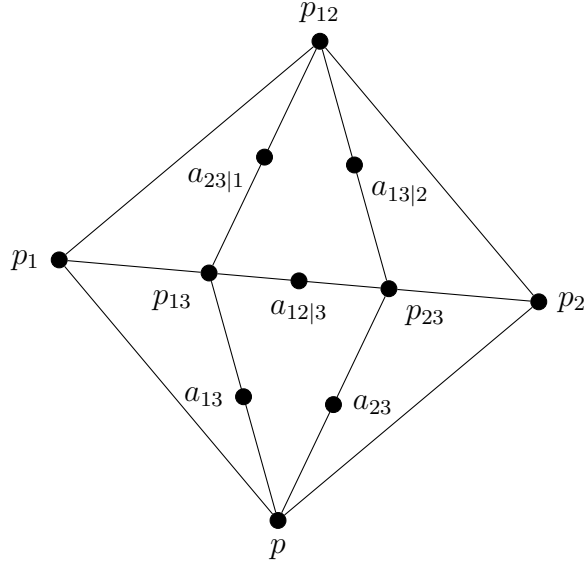

\begin{example} \label{ex:NObody}
Consider the  {\bf app} triangle $\{a_{12}, p_{3}, p_{123} \}$. The corresponding
$7$-dimensional cone in ${\rm trop}(V(J_3)) \subset \RR^{14}$ consists of
all vectors $\nu = \mu+(a, 0 ,0,0, 0, 0, 0, 0, 0, b, 0, 0, 0, c)$,
where $a,b,c > 0$ and $\mu $ is in the subspace $L$ in (\ref{eq:lineality}).
Each of these $\nu$ is a valuated gaussoid.

The initial ideal ${\rm in} _\nu(J_3)$ is obtained by setting
$a_{12}$, $p_3$ and $p_{123}$ to zero in all $21$ quadrics in Example
\ref{ex:J3}. The resulting ideal is generated by binomials and is prime.
Hence, ${\rm in} _\nu(J_3)$ is a toric ideal.  In the language of
\cite{KM}, the cone indexed by $\{a_{12}, p_{3}, p_{123} \}$ is a
prime cone.

The vector $\nu$ defines a degeneration of the Lagrangian Grassmannian
${\rm LGr}(3,6) = V(J_3)$ to the toric variety $V({\rm
in}_\nu(J_3))$. Both are $6$-dimensional and have degree $16$.  The
corresponding lattice polytope is the Newton--Okounkov body. It has
dimension $6$ and volume $16$.  It is the join of the triangle
$\{a_{12}, p_{3}, p_{123} \}$ with the $3$-dimensional polytope
shown in Figure~\ref{fig:zwei}. This polytope has $6$ vertices, $11$
edges and $7$ facets. Five additional points lie on edges.  The toric
ideal for this configuration of $11 = 6+5$ lattice points in $3$-space
is equal to ${\rm in}_\nu(J_3)$.  \hfill $\diamondsuit$
\end{example}

The positive part of the tropical Grassmannian plays an important role
in the theory of cluster algebras \cite{BCL, SW}.  Note that
$\,{\rm trop}_+({\rm Gr}(3,6))\,$ was worked out in~\cite[\S~6]{SW}:
it is the boundary of a $4$-polytope known as the
$D_4$-associahedron. In what follows we determine the analogue for the
Lagrangian Grassmannian, that is, the space of positive valuated
gaussoids
\begin{equation}
\label{eq:tropplus}
{\rm trop}_+({\rm LGr}(3,6)) \,\, = \,\,
{\rm trop}_+({\rm Gr}(3,6)) \,\, \cap \,\,
{\rm trop}({\rm LGr}(3,6)) .
\end{equation}

\begin{corollary} \label{cor:1016} The intersection {\rm
(\ref{eq:tropplus})} corresponds to a triangulated $2$-sphere with
$10$ vertices, $24$ edges and $16$ facets.  It is the boundary of the
simplicial $3$-polytope shown in Figure \ref{fig:tensixteen}.
\end{corollary}

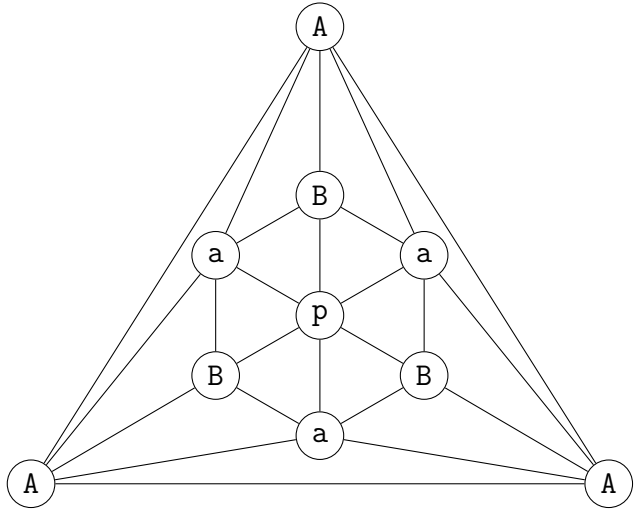
\begin{figure}[h]
\begin{center}
\begin{tikzpicture}[scale=.8]
\tikzset{vertex/.style = {shape=circle,draw,minimum size=18pt, inner sep=2pt}}
\tikzset{edge/.style = {line width=2pt}}

\node [vertex] (p) at (0,0) {\tt{p}};
\node [vertex] (adown) at (0,-2) {\tt{a}};
\node [vertex] (aleft) at (-1.732,1) {\tt{a}};
\node [vertex] (aright) at (1.732,1) {\tt{a}};
\node [vertex] (Bup) at (0,2) {\tt{B}};
\node [vertex] (Bleft) at (-1.732,-1) {\tt{B}};
\node [vertex] (Bright) at (1.732,-1) {\tt{B}};
\node [vertex] (Aup) at (0,4.8) {\tt{A}};
\node [vertex] (Aleft) at (-4.8,-2.8) {\tt{A}};
\node [vertex] (Aright) at (4.8,-2.8) {\tt{A}};

\draw (Aup) -- (Aleft) -- (Aright) -- (Aup);
\draw (Bup) -- (aright) -- (Bright) -- (adown) -- (Bleft) -- (aleft) -- (Bup);
\draw (p) -- (Bup);
\draw (p) -- (Bleft);
\draw (p) -- (Bright);
\draw (p) -- (adown);
\draw (p) -- (aleft);
\draw (p) -- (aright);
\draw (Aup) -- (aleft) -- (Aleft);
\draw (Aup) -- (aright) -- (Aright);
\draw (Aleft) -- (adown) -- (Aright);
\draw (Bleft) -- (Aleft);
\draw (Bright) -- (Aright);
\draw (Bup) -- (Aup);
\end{tikzpicture}
\vspace{-0.1in}
\end{center}
 \caption{The $3$-polytope that represents
 all positive valuated matroids for $n=3$.
 Its boundary, a simplicial $2$-sphere, is the
 positive tropical Lagrangian Grassmannian $ {\rm trop}_+({\rm LGr}(3,6))$. }
 \label{fig:tensixteen}
\end{figure}

 \begin{proof}[Proof and Explanation]
We examined all $153$ maximal cones in Theorem \ref{thm:J3tropical}.
A cone lies in (\ref{eq:tropplus}) if and only if its
initial ideal ${\rm in}_\nu(J_3)$ is generated by
pure difference binomials $m_1 - m_2$. This happens
for the following $16$ cones.  For each of them, we list a representative vector $\nu$:
$$
{\tt apB} /  {\rm EEEG} \qquad  \qquad \quad
{\tt aAB} / {\rm EEFG} \qquad \qquad \quad
{\tt aAA} /  {\rm EEFFb} \qquad \quad\,
{\tt AAA} / {\rm  FFFGG}
$$
$$
\begin{bmatrix}
 2 2 2 5 0 0 8 4 0 0 0 0 0 0 \\
 2 5 2 2 0 0 8 4 0 0 0 0 0 0 \\
 2 2 0 0 2 5 8 0 0 0 0 4 0 0 \\
 2 5 0 0 2 2 8 0 0 0 0 4 0 0 \\
 0 0 2 2 2 5 8 0 0 0 0 0 0 4 \\
 0 0 2 5 2 2 8 0 0 0 0 0 0 4
 \end{bmatrix}
\quad
\begin{bmatrix}
 0 0 6 6 6 9 4 0 0 0 0 0 0 4 \\
 0 0 6 9 6 6 4 0 0 0 0 0 0 4 \\
 6 6 0 0 6 9 4 0 0 0 0 4 0 0 \\
 6 6 6 9 0 0 4 4 0 0 0 0 0 0 \\
 6 9 0 0 6 6 4 0 0 0 0 4 0 0 \\
 6 9 6 6 0 0 4 4 0 0 0 0 0 0 \\
 \end{bmatrix}
\quad
\begin{bmatrix}
 4 4 2 2 6 9 0 0 0 0 0 0 0 0 \\
 4 4 6 9 2 2 0 0 0 0 0 0 0 0 \\
 6 9 4 4 2 2 0 0 0 0 0 0 0 0 \\
 \end{bmatrix}
\quad
 \begin{bmatrix}
7 7 5 5 6 6 0 0 0 0 0 0 0 0
\end{bmatrix}
$$
For instance,  the vector $\nu = ( 2 2 2 5 0 0 8 4 0 0 0 0 0 0) $,
indexed as in (\ref{eq:fourteen}), is a positive valuated gaussoid.
It lies in a cone of type ${\tt apB}$, and hence in a cone of type EEEG in ${\rm trop}_+({\rm Gr}(3,6))$.
Each positive valuated gaussoid $\nu$ records the $\epsilon$-orders of the
principal and almost-principal minors of a covariance matrix that defines
a Gaussian ${\rm MTP}_2$ distribution  over $\RR \{\! \{ \epsilon \} \! \}$.

For example, $\nu = (7 7 5 5 6 6 0 0 0 0 0 0 0 0)$ is realized in this sense by the covariance matrix
$$ \Sigma \,\, = \,\, \hbox{the inverse of} \,\,
\begin{pmatrix}
1 & - \epsilon^7 & -\epsilon^5 \\
 -\epsilon^7 & 1 & -\epsilon^6 \\
-\epsilon^5 & -\epsilon^6 & 1
\end{pmatrix}.
$$
Figure \ref{fig:tensixteen} is a combinatorial classification of
all Gaussian ${\rm MTP}_2$ distributions over $\RR \{ \! \{ \epsilon \} \! \}$.
\end{proof}

\section{Realizability}\label{s:realizable}

In this section we study the realizability problem for gaussoids and
oriented gaussoids.  There is a substantial literature on the
realizability of matroids and oriented matroids.  We point
to~\cite{FMM} and the references therein.  It is our aim to extend
this to the setting developed in this paper.  Our first result
concerns the realizability of uniform oriented gaussoids for $n=4$.

\begin{theorem}\label{thm:classifyOriented4}
There are $46$ symmetry classes of uniform oriented gaussoids for
$n=4$, listed in Table~\ref{tab:census4}.  All but one of them are
realizable.  The unique non-realizable class admits a bi-quadratic
final polynomial in the sense of Bokowski and Richter {\rm
\cite{BoRG}}.
\end{theorem}

\begin{table}[t]
\centering
\begin{tiny}
\[
\begin{matrix}
1 &  ++++++++++++++++++++++++ & (1/8, 1/16, 1/4, 1/4, 1/16, 1/8) \\
2 &  ++++++++++++------------ & (1/8, 1/16, 1/4, -1/2, -1/4, -1/16) \\
3 &  +++++++++++++-+---+----- & (1/4, 1/2, 1/16, 1/32, -1/128, -1/2) \\
4 &  +++++++++++++++++++++++- & 1/100 \cdot (44, 50, 51, 50, 51, 30)\\
5 &  ++++++++++-++-++----++++ & (1/4, 1/8, 1/512, 1/128, -1/4, 1/8) \\
6 &  ++++++++++--+-+----+++++ & (1/4, 1/2, 1/128, 1/64, -1/32, 1/4) \\
7 &  +++++++++++++++-+-+-+-+- & 1/100 \cdot (57, 57, 76, 39, 12, 12)\\
8 &  +++++++++++++++++-+++--- & 1/100 \cdot (64, 55, 60, 76, 32, 6)\\
9 &  +++++++++++++++-+-+----- & (1/16, 1/32, 1/2, 1/64, 1/128, -1/2) \\
10 &  +++++++++++++++++-+++-+- & 1/100 \cdot (45, 57, 66, 57, 26, 19)\\
11 &  ++++++++++-++-+++---++++ & 1/100 \cdot (75, 69, 45, 45, 7, 75)\\
12 &  +++++++++++++++---+----- & (1/16, 1/8, 1/2, 1/32, -1/512, -1/2) \\
13 &  ++++++++++--+---++++++++ & (1/16, 1/2, 1/8, 1/512, 1/64, 1/2) \\
14 &  ++++++++++--+-+-----++++ & (1/8, 1/2, 1/16, 1/64, -1/64, 1/4) \\
15 &  +++++++++++-+---+++++-++ & 1/100 \cdot (53, 76, 46, 8, 71, 27)\\
16 &  ++++++++++--++--++++++++ & (1/64, 1/8, 1/128, 1/128, 1/8, 1/8) \\
17 &  +++++++++++++-+-+-+-+--- & (1/2, 1/8, 1/8, 1/64, 1/32, 1/4096) \\
18 &  +++++++++++++-+-+-+-+-+- & (1/16, 1/2, 1/2, 1/64, 1/128, 1/8) \\
19 &  +++++++++++++++--------- & (1/2, 1/256, 1/64, 1/256, -1/128, -1/4) \\
20 &  ++++++++++-++-+-+---++++ & 1/100 \cdot (85, 67, 37, 39, 6, 64)\\
21 &  +++++++++++++++++-+----- & (1/4, 1/16, 1/2, 1/16, 1/32, -1/16) \\
22 &  +++++++++++-+---++++++++ & 1/100 \cdot (59, 59, 48, 13, 59, 59)\\
23 &  +++++++++++++++++-++---- & (1/2, 1/2, 1/8, 1/2, 1/256, -1/8) \\
24 &  ++++++++++-++-+-+---+++- & 1/100 \cdot (81, 84, 39, 43, 2, 49)\\
25 &  +++++++++++-+-+-+++++-++ & 1/100 \cdot (60, 85, 39, 21, 55, 27)\\
26 &  +++++++++++++++++-+-+-+- & (1/8, 1/16, 1/8, 1/16, 1/128, 1/512) \\
27 &  ++++++++++--+-+-++-+++++ & (1/4, 1/2, 1/1024, 1/128, 1/1024, 1/4) \\
28 &  +++++++++++-+++-++++++++ & 1/100 \cdot (58, 58, 41, 41, 58, 58)\\
29 &  +++++++++++++++++++++--- & (1/16, 1/8, 1/8, 1/4, 1/4, 1/256) \\
30 &  ++++++++++-+--+-----+++- & (1/16, 1/8, 1/16384, -1/8192, -1/4, 1/512) \\
31 &  ++++++++++++++++++++---- & (1/16, 1/4, 1/8, 1/8, 1/8, -1/32) \\
32 &  ++++++++++++++---------- & (1/8, 1/32, 1/2, 1/32, -1/8, -1/2) \\
33 &  ++++++++++--+-+-+---++++ & 1/100 \cdot (63, 70, 32, 33, 4, 63)\\
34 &  +++++++++---+---++++++++ & (1/4, 1/32, 1/256, 1/256, 1/2, 1/2) \\
35 &  +++++++++++++-+-+-+----- & (1/2, 1/4, 1/4, 1/16, 1/32, -1/4) \\
36 &  ++++++++++++++--+-+-+-+- &  \hbox{bi-quadratic final polynomial} \\
37 &  +++++++++++++----------- & (1/2, 1/2, 1/32, 1/128, -1/2, -1/16) \\
38 &  +++++++++++++++++-+-+--- & (1/4, 1/4, 1/16, 1/2, 1/256, 1/1024) \\
39 &  +++++++-++-++-+++---++++ & 1/100 \cdot (83, 46, 33, 33, 5, 83)\\
40 &  ++++++++++--+---+-++++++ & (1/4, 1/2, 1/32, 1/32768, 1/2048, 1/4) \\
41 &  +++++++++++-++--++++++++ & 1/100 \cdot (46, 46, 43, 30, 74, 74)\\
42 &  ++++++++++++++++++++++-- & (1/8, 1/256, 1/8, 1/64, 1/2, 1/256) \\
43 &  ++++++++++--+-+-+--+++++ & (1/4, 1/32, 1/512, 1/512, 1/4096, 1/2) \\
44 &  +++++++++++++++-+-+-+--- & 1/100 \cdot (73, 59, 71, 49, 25, 6)\\
45 &  ++++++++++-++-+-----++++ & (1/4, 1/2, 1/64, 1/128, -1/4, 1/8) \\
46 &  ++++++++++-++-+-----+++- & (1/2, 1/2, 1/1024, 1/256, -1/2, 1/8)
\end{matrix} \vspace{-0.22in}
\]
\end{tiny}
\caption{The $46$ symmetry classes of uniform oriented gaussoids for
$n=4$.\label{tab:census4}}
\end{table}

\begin{proof}[Proof and Explanation]
The $46$ classes were derived from the list of $5376$ uniform oriented
gaussoids in Theorem~\ref{thm:oriegaussoidcensus}.  The lists of $24$
signs in Table~\ref{tab:census4} is with respect to the ordering
$$ \begin{matrix} a_{12}, a_{12|3}, a_{12|4}, a_{12|34},  a_{13}, a_{13|2}, a_{13|4}, a_{13|24},
   a_{14}, a_{14|2}, a_{14|3}, a_{14|23},  \\  a_{23}, a_{23|1}, a_{23|4}, a_{23|14},
   a_{24}, a_{24|1}, a_{24|3}, a_{24|13},   a_{34}, a_{34|1}, a_{34|2}, a_{34|12} . \end{matrix}  $$
In each realizable case, we list the entries
$(\sigma_{12},\sigma_{13},\sigma_{14},\sigma_{23},\sigma_{24},\sigma_{34})$
of a positive definite symmetric $4 \times 4$-matrix $\Sigma$ with
$\sigma_{11} = \sigma_{22} = \sigma_{33} = \sigma_{44} = 1$ for that
oriented gaussoid.  The realization space of an oriented gaussoid is a
semi-algebraic set.  We used random search with values $2^{-k}$ for
small $k$ and the optimization software SCIP~\cite{scip} to find
realizations.

The oriented gaussoid \#36 is of special interest since it has a {\em
bi-quadratic final polynomial}. We review this concept from
\cite{BoRG}.  The edge trinomials can be written as
$ x_1 x_2 + x_3 x_4 - x_5 x_6=0 $, where each $x_i$ is a positive
unknown, equal to either some $p_I$ or some
$a_{ij|K}$ multiplied by its sign.
The equation hence implies the inequalities $\,x_1 x_2 < x_5 x_6$ and
$x_3 x_4 < x_5 x_6$.  After replacing each $x_i$ by its logarithm,
$y_i = {\rm log}(x_i)$, we get $y_1 + y_2 < y_5+y_6 $ and
$ y_3+y_4 < y_5 + y_6$.  Using Linear Programming (LP), we can easily
decide whether the resulting system of linear inequalities has a
solution. If not, then the oriented gaussoid is non-realizable.
A~solution to the dual LP yields a non-realizability certificate known
as bi-quadratic final polynomial.

Here is how it works for type \#36. Among the edge trinomials we find
the following:
\[
\begin{matrix} (-a_{23|4}) (p_{134})  +  (a_{13|4})(a_{12|34}) - (-a_{23|14})( p_{34}) \, ,\,\,
(a_{12|3}) (p_{134})  +  (a_{14|3})  (-a_{24|13}) - (a_{12|34})  (p_{13})     \\
(a_{23|1}) ( p_{134}) + (-a_{23|14}) (p_{13}) - (-a_{34|1})  (-a_{24|13})   \, , \,\,
(a_{34})( p_{134} ) +  (-a_{34|1})  (p_{34} ) - (a_{13|4}) ( a_{14|3}).
\end{matrix}
\]
These are elements of $J_4$, written in such a way that each parenthesis is positive for \#36.
From these four equations we infer the following inequalities among positive quantities:
\[
\begin{matrix}
     a_{13|4}  a_{12|34} &<& (-a_{23|14})  p_{34} & \qquad &
  a_{14|3} (-a_{24|13}) &<& a_{12|34}  p_{13} \\
    (-a_{23|14})  p_{13} &<& (-a_{34|1})  (-a_{24|13}) &  \qquad &
    (-a_{34|1})  p_{34}  &<&  a_{13|4}  a_{14|3}
    \end{matrix}
\]
The product of the left hand sides equals the product of the right
hand sides.
\end{proof}

We now briefly discuss the case $n=5$.  A {\em complex
realization} of a gaussoid $\mathcal{G}$ on $[n]$ is a symmetric
$n \times n$-matrix $\Sigma$ with entries in $\CC$ whose principal
minors are nonzero and whose vanishing almost-principal minors are
indexed by $\mathcal{G}$.   The following example can be viewed as a gaussoid
analog to the V\'amos matroid, which is the smallest non-realizable matroid.

\begin{example}\label{ex:VamosGaussoid}
Let $n=5$. The following collection of ten $2$-faces of the
$5$-cube is a gaussoid:
\[
\mathcal{G} \,= \, \bigl\{
a_{12},\,a_{13|4},\,a_{14|5},\,a_{15|23},\, a_{23|5},\, a_{24|135},\,
a_{25|34},\, a_{34|12},\, a_{35|1},\, a_{45|2} \,\bigr\}.
\]
To see that $\mathcal{G}$ is not realizable over $\CC$, consider the
ideal in $\QQ[\sigma_{12}, \sigma_{13}, \ldots,\sigma_{45}]$ generated
by these $10$ almost-principal minors, for a symmetric
$5 \times 5$-matrix with ones on the diagonal and unknowns
$\sigma_{ij}$ off the diagonal.  Saturation with respect to
$p_{24} = 1 - \sigma_{24}^2$ yields the maximal ideal
$\langle \sigma_{12}, \sigma_{13}, \ldots,\sigma_{45} \rangle$.  
This implies that  there is no complex symmetric $5 \times 5$-matrix  $(\sigma_{ij})$
with $\sigma_{13} p_{24} \not= 0$ for which all the $10$ minors 
in $\mathcal{G}$ are zero. \hfill$\diamondsuit$
\end{example}

With Example~\ref{ex:VamosGaussoid} it is now easy to define a
non-realizable valuated gaussoid.

\begin{example}\label{ex:valuatedNonRealizable}
Fix $\mathcal{G}$ as in Example~\ref{ex:VamosGaussoid}.  Let $\nu $ be
the map from $\mathcal{A} \cup \mathcal{P}$ to $\RR$ that takes
$\mathcal{G}$ to $1$ and
$(\mathcal{A}\backslash \mathcal{G}) \cup \mathcal{P} $ to~$0$.  By
examining all the edge and square trinomials in $(J_5)_2$, we can
verify that $\nu$ is a valuated gaussoid. However, it is not
realizable.  There is no point in $V(J_5)$ over the Puiseux series
field $\CC \{ \!\{ \epsilon \} \! \}$ whose coordinates have valuation
$\nu$.  Such a point would come from a symmetric matrix $\Sigma$ whose
entries are in $\CC \{ \!\{ \epsilon \} \! \}$ and have valuations
$\geq 0$.  Setting $\epsilon = 0$ in that matrix gives a complex
realization of~$\mathcal{G}$. But this does not exist.  \hfill
$\diamondsuit$
\end{example}

In the preprint version of this article we conjectured that
the non-realizable valuated gaussoid above is minimal.
In other words, we conjectured that 
all tropical gaussoids for $n =4$ are realizable, i.e.~the
square trinomials and edge trinomials are a tropical basis for~$J_4$.

This conjecture is false. It was disproved by G\"orlach, Ren and Sommars,
using their new algorithm for tropical basis verification \cite{GRS}.
Here is one of the explicit examples they found.

\begin{theorem}[G\"orlach et al.~\cite{GRS}] \label{thm:gaussoids}
  There exist non-realizable valuated gaussoids for $n=4$.
  \end{theorem}
  
\begin{proof}
We order the $40$ elements of $\mathcal{P} \cup \mathcal{A}$ as follows:
 $$ \begin{small} \begin{matrix}
 p_{\emptyset},p_{1},p_{12},p_{123},p_{1234},p_{124},p_{13},p_{134},
p_{14},p_{2},p_{23},p_{234},p_{24},p_{3},p_{34},p_{4} ,
  a_{12},a_{12|3},a_{12|34},a_{12|4}, a_{13},a_{13|2} ,\\
  a_{13|24},a_{13|4},a_{14},a_{14|2}, a_{14|23},a_{14|3},
  a_{23},a_{23|1},a_{23|14},a_{23|4},a_{24},a_{24|1},a_{24|13},a_{24|3},a_{34},a_{34|1},a_{34|12},a_{34|2} .
  \end{matrix} \end{small} $$
Let $\nu$ be the map $\mathcal{P} \cup \mathcal{A} \rightarrow \mathbb{R}$
that takes the following values, listed in the order above:
   $$ \! (
  14, 10, 6, 0, 6, 8, 8, 2, 8, 6, 6, 2, 8, 8, 8, 8, 8, 4, 2, 10, 9, 3, 
    5, 5,   9, 11, 1, 5, 7, 5, 5, 5, 7, 7, 1, 5, 8, 6, 4, 4).
  $$
  One can check that $\nu$ is a valuated gaussoid, i.e.~if $f$ is any of the
  square trinomials or edge trinomials in $J_4$ then
  ${\rm in}_\nu(f)$  is not a monomial.   On the other  hand,   
  the initial ideal   ${\rm in}_\nu(J_4)$ contains the monomial $a_{23}a_{23|1}$.
  Hence the valuated gaussoid $\nu$ is not realizable.
\end{proof}

We close the paper with two open problems concerning the
realizability of gaussoids. 
Realization problems can be formulated as feasibility problems of
(semi-)algebraic sets.  The following refers to Theorem
\ref{thm:gaussoidcensus}. It is a challenge as far as computation
goes, but it is also an excellent opportunity for gaining statistical
insights about Gaussian random variables.

\begin{challenge} Classify the $16981$ $(\ZZ/2\ZZ)^n\rtimes S_n$-orbits
of gaussoids for $n=5$ according to their realizability
over~$\CC$.  Classify all $254826$ $\ZZ/2\ZZ \rtimes S_n$-orbits
according to realizability.
\end{challenge}

The {\em Universality Theorem} due to Mn\"ev \cite[\S 8.6]{OM}
states, roughly speaking, that any variety arises as the
realization space of a matroid, and any semialgebraic set
arises as the realization space of an oriented matroid.
We wonder whether the same is true for gaussoids.

\begin{problem}
Does universality hold for gaussoids? Can arbitrary varieties and
arbitrary semialgebraic sets be the realization spaces of gaussoids
and oriented gaussoids respectively?
\end{problem}

\bigskip \bigskip
\bigskip

\begin{small}
\noindent {\bf Acknowledgements.} We thank Moritz Firsching, Paul
G\"orlach, Jon Hauenstein, Mateusz Micha{\l}ek, Peter Nelson, Yue Ren,
Caroline Uhler and Charles Wang for help with this project.  Bernd
Sturmfels was partially supported by the Einstein Foundation Berlin
and the US National Science Foundation (DMS-1419018, DMS-1440140).
Tobias Boege and Thomas Kahle were partially supported by the Deutsche
Forschungsgemeinschaft (314838170, GRK 2297, ``MathCoRe'').
\end{small}

\bigskip \medskip

\noindent
\footnotesize {\bf Authors' addresses:}

\smallskip

\noindent Tobias Boege, OvGU Magdeburg, Germany,
{\tt tboege@st.ovgu.de}

\noindent Alessio D'Al\`i,  Max-Planck Institute for Math in the Sciences, Leipzig, Germany,
{\tt alessio.dali@mis.mpg.de}

\noindent Thomas Kahle, OvGU Magdeburg, Germany,
{\tt thomas.kahle@ovgu.de}

\noindent Bernd Sturmfels,
 \  MPI-MiS Leipzig, {\tt bernd@mis.mpg.de} \ and \
UC  Berkeley,  {\tt bernd@berkeley.edu}
\end{document}